\documentclass[a4paper,11pt]{article}
\usepackage{adjustbox}
\usepackage{aligned-overset}
\usepackage{amsmath,amsthm}
\usepackage{authblk}
\usepackage[style=numeric-comp,maxbibnames=99,bibencoding=utf8,giveninits=true,backend=biber]{biblatex}
\usepackage{booktabs}
\usepackage{longtable}
\usepackage{multirow}
\usepackage{cancel}
\usepackage{cases}
\usepackage{colortbl}
\usepackage[hmargin={26mm,26mm},vmargin={30mm,35mm}]{geometry}
\usepackage{nicefrac}
\usepackage[colorlinks,allcolors={blue}]{hyperref}
\usepackage{paralist}
\usepackage{subcaption}
\usepackage{enumitem}
\usepackage[compact,small]{titlesec}
\usepackage{tikz-cd}
\usetikzlibrary{arrows}
\usepackage{multicol}
\usepackage{comment}

\usepackage{newtxtext}
\usepackage{newtxmath}

\bibliography{rm}

\newcommand{\email}[1]{\href{mailto:#1}{#1}}

\numberwithin{equation}{section}

\newcommand{\adj}{\ast}

\newtheorem{theorem}{Theorem}
\newtheorem{proposition}[theorem]{Proposition}
\newtheorem{lemma}[theorem]{Lemma}

\theoremstyle{remark}
\newtheorem{remark}[theorem]{Remark}
\theoremstyle{definition}

\newcommand{\Real}{\mathbb{R}}

\DeclareRobustCommand{\bvec}[1]{\boldsymbol{#1}}
\newcommand{\cvec}[1]{\bvec{\mathcal{#1}}}

\newcommand{\Cspace}[1]{\mathcal C_{#1}}
\newcommand{\HSob}{H_{1}(\Omega)^2 \times H_1 (\Omega)}
\newcommand{\HSobz}{H_{1,0}(\Omega)^2 \times  H_{1,0}(\Omega)}

\DeclareMathOperator{\GRAD}{\bf grad}
\DeclareMathOperator{\CURL}{\bf curl}
\DeclareMathOperator{\DIV}{div}
\DeclareMathOperator{\ROT}{rot}
\DeclareMathOperator{\VROT}{\bf rot}
\DeclareMathOperator{\hess}{\boldsymbol{\mathcal{H}}}

\newcommand{\compl}{{\rm c}}

\newcommand{\symbolproj}{\pi}
\newcommand{\lproj}[2]{\symbolproj_{\mathcal{P},#2}^{#1}}
\newcommand{\vlproj}[2]{\bvec{\symbolproj}_{\cvec{P},#2}^{#1}}
\newcommand{\Rproj}[2]{\bvec{\symbolproj}_{\cvec{R},#2}^{#1}}
\newcommand{\Rcproj}[2]{\bvec{\symbolproj}_{\cvec{R},#2}^{\compl,#1}}

\newcommand{\edges}[1]{\mathcal{E}_{#1}}
\newcommand{\vertices}[1]{\mathcal{V}_{#1}}

\newcommand{\ET}{\edges{T}}

\newcommand{\VT}{\vertices{T}}

\newcommand{\VE}{\vertices{E}}

\newcommand{\normal}{\bvec{n}}
\newcommand{\tangent}{\bvec{t}}

\newcommand{\letterPoly}{\mathcal{P}}
\newcommand{\Poly}[1]{\letterPoly^{#1}}
\newcommand{\Roly}[1]{\cvec{R}^{#1}}
\newcommand{\cRoly}[1]{\cvec{R}^{\compl,#1}}

\newcommand{\stokes}{2}
\newcommand{\ddr}{1}

\newcommand{\norm}[2]{\|#2\|_{#1}}
\newcommand{\seminorm}[2]{|#2|_{#1}}
\newcommand{\vvvert}{\vert\kern-0.25ex\vert\kern-0.25ex\vert}

\newcommand{\Th}{\mathcal{T}_h}
\newcommand{\Eh}{\mathcal{E}_h}
\newcommand{\Ehi}{\mathcal{E}_{h}^{\text{i}}}
\newcommand{\Vh}{\mathcal{V}_h}

\newcommand{\Pgrad}[2]{\bvec{P}_{\GRAD,#2}^{#1}}
\newcommand{\Prot}[2]{\bvec{P}_{\VROT,#2}^{#1}}

\newcommand{\SPgrad}{P_{\stokes,T}^{k+1}}
\newcommand{\SPgradh}{P_{\stokes,h}^{k+1}}

\newcommand{\SProt}{\bvec{P}_{\ROT,T}^{k}}

\newcommand{\XSgrad}[1]{\underline{H}_{2}^{k}(#1)}
\newcommand{\XSgrado}[1]{\underline{H}_{2,1,0}^{k}(#1)}

\newcommand{\XSrot}[1]{\underline{H}_{1}^{k+1}(#1)^2}

\newcommand{\ISgrad}{\underline{I}_{\stokes,h}^{k}}

\newcommand{\SGRAD}[1]{\underline{\bvec{G}}^{k}_{2,#1}}

\newcommand{\Gqen}{G^{\normal}_{q,E}}
\newcommand{\Gqv}{\bvec G_{q,V}}
\newcommand{\Gqet}{G^{\tangent}_{2,E}}
\newcommand{\nablaT}{\bvec{G}_{2,T}^{k}}
\newcommand{\nablaTfull}{\bvec{G}_{2,T}^{k}}

\newcommand{\gammaS}{\gamma^{k+1}_{\stokes,\partial T}}

\newcommand{\Id}{\mathrm{Id}}
\newcommand{\sskw}{\mathop{\mathrm{sskw}}}

\newcommand{\XdRgrad}[1]{\underline{H}_{1}^{k+1}(#1)}
\newcommand{\XdRrot}[1]{\underline{\boldsymbol{H}}_{\operatorname{\boldsymbol{\mathrm{rot}}}}^{k+1}(#1)}
\newcommand{\XdRroto}[1]{\underline{\boldsymbol{H}}_{\operatorname{\boldsymbol{\mathrm{rot}}},0}^{k+1}(#1)}

\newcommand{\tXdRgrad}[1]{\XdRgrad{#1}^2}
\newcommand{\tXdRgrado}[1]{\underline{H}_{1,0}^{k+1}(#1)^2}
\newcommand{\tXdRrot}[1]{\XdRrot{#1}}
\newcommand{\tXdRroto}[1]{\XdRroto{#1}}

\newcommand{\tGRAD}[1]{\underline{\bvec G}_{1,#1}^{k+1}}
\newcommand{\tGRADs}[1]{\underline{\bvec G}_{1,s,#1}^{k+1}}

\newcommand{\IdRgrad}{\underline{\bvec{I}}_{\ddr,h}^{k+1}}
\newcommand{\IdRgradT}{\underline{\bvec{I}}_{\ddr,T}^{k+1}}
\newcommand{\IdRrot}{\underline{\bvec{I}}_{\operatorname{\boldsymbol{\mathrm{rot}}},h}^{k+1}}

\newcommand{\tGsTfull}{\bvec{G}_{1,s,T}^{k+1}}
\newcommand{\tGshfull}{\bvec{G}_{1,s,h}^{k+1}}
\newcommand{\tGT}{\bvec{G}_{1,T}^{k+1}}
\newcommand{\tGTfull}{\bvec{G}_{1,T}^{k+1}}

\newcommand{\tGE}{\bvec{G}^{k+1}_{1,E}}

\newcommand{\gammadR}[1]{\bvec{\gamma}^{k+2}_{\ddr,#1}}

\newcommand\tr{\operatorname{tr}}

\newcommand{\Hess}[1]{\underline{\bvec{\mathcal{H}}}^{k+1}_{#1}}
\newcommand{\Hessh}{\bvec{\mathcal{H}}^{k+1}_{h}}

\newcommand{\dof}[2]{\underline{#1}_{#2}}
\newcommand{\dofh}[1]{\underline{#1}_h}
\newcommand{\dofT}[1]{\underline{#1}_T}

\newcommand{\const}[2]{\mathcal{E}_h((\bvec{\psi},u),(#1,#2))}

\newcommand{\injgrad}[1]{i_{\GRAD,#1}}

\usepackage{pgfplots,pgfplotstable}
\usepackage{tikz-cd}
\graphicspath{{figures/}}

\newcommand{\logLogSlopeTriangle}[5]
{
    \pgfplotsextra
    {
        \pgfkeysgetvalue{/pgfplots/xmin}{\xmin}
        \pgfkeysgetvalue{/pgfplots/xmax}{\xmax}
        \pgfkeysgetvalue{/pgfplots/ymin}{\ymin}
        \pgfkeysgetvalue{/pgfplots/ymax}{\ymax}

        \pgfmathsetmacro{\xArel}{#1}
        \pgfmathsetmacro{\yArel}{#3}
        \pgfmathsetmacro{\xBrel}{#1-#2}
        \pgfmathsetmacro{\yBrel}{\yArel}
        \pgfmathsetmacro{\xCrel}{\xArel}

        \pgfmathsetmacro{\lnxB}{\xmin*(1-(#1-#2))+\xmax*(#1-#2)} 
        \pgfmathsetmacro{\lnxA}{\xmin*(1-#1)+\xmax*#1} 
        \pgfmathsetmacro{\lnyA}{\ymin*(1-#3)+\ymax*#3} 
        \pgfmathsetmacro{\lnyC}{\lnyA+#4*(\lnxA-\lnxB)}
        \pgfmathsetmacro{\yCrel}{(\lnyC-\ymin)/(\ymax-\ymin)}

        \coordinate (A) at (rel axis cs:\xArel,\yArel);
        \coordinate (B) at (rel axis cs:\xBrel,\yBrel);
        \coordinate (C) at (rel axis cs:\xCrel,\yCrel);

        \draw[#5]   (A)-- node[pos=0.5,anchor=north] {\scriptsize{1}}
                    (B)-- 
                    (C)-- node[pos=0.,anchor=west] {\scriptsize{#4}} 
                    cycle;
    }
}

\begin{document}

\title{An arbitrary-order BGG-based discrete scheme for the Reissner--Mindlin plate problem on polygonal meshes}

\author[1]{Arax Leroy}
\affil[1]{IMAG, Univ. Montpellier, CNRS, Montpellier, France, \email{arax.leroy@umontpellier.fr}}
\maketitle

\begin{abstract}
We design and analyse an arbitrary-order numerical scheme for the Reissner--Mindlin plate problem on general polygonal meshes. The scheme is derived from the Hodge--Laplacian associated with a discrete Bernstein--Gelfand--Gelfand (BGG) twisted complex and exploits a discrete $H_2$-based construction for the transverse displacement. We establish a discrete Korn inequality for the underlying Discrete de Rham method (DDR) spaces and prove the convergence of the method. At the lowest order, the analysis yields an error estimate that is uniform with respect to the plate thickness, showing that the scheme is locking-free. Numerical experiments on several families of polygonal meshes support the theoretical results and illustrate the benefits of the enhanced continuity of the transverse displacement discretisation.
\end{abstract}

\section{Introduction}
\label{sec:introduction}

The Reissner--Mindlin model, introduced in the seminal works \cite{Reissner:45,Mindlin:51}, is a classical model for the bending of moderately thick elastic plates. Its displacement--rotation formulation accounts for transverse shear effects while involving only first-order derivatives, which makes it particularly suitable for numerical approximation.

A central challenge is to design discretisations that remain stable and accurate uniformly with respect to the plate thickness $t$. As this thickness tends to zero, the shear term enforces the Kirchhoff constraint relating the rotation to the gradient of the transverse displacement. Discrete spaces that do not reproduce this constraint appropriately may become excessively
stiff, leading to the well-known shear-locking phenomenon. A method is said to be locking-free when its stability and error estimates remain uniform with respect to $t$. The construction of such methods has been the subject of a large literature, see, among others, \cite{Brezzi.Fortin:86,Arnold.Falk:89,Brezzi.Bathe.Fortin:89,Gallistl.Schedensack:21}.

Finite element methods have been historically preferred for simulating solid mechanics models, such as plate problems; however, they present restrictions on the kinds of meshes they support -- typically conforming triangular or quadrangular meshes -- which limits their flexibility. By contrast, methods that are applicable on generic polygonal meshes present appealing benefits, such as seamless local mesh refinement (with hanging nodes) to capture steep solutions or complex geometry, or easy mesh agglomeration for multi-grid algorithms.
Several polygonal methods have been proposed for the Reissner--Mindlin problem. Virtual Element Method (VEM) formulations include an MITC-type construction \cite{Chinosi:18} and a shear--deflection formulation whose error estimates are uniform with respect to the plate thickness \cite{Beirao-da-Veiga.Mora.ea:19*1}. 
A locking-free Weak Galerkin method on general polygonal meshes was introduced in
\cite{Ye.Zhang.Zhang:20}. More recently, an arbitrary-order scheme combining the Discrete de Rham (DDR) and Hybrid High-Order methods was proposed and analysed in \cite{Di-Pietro.Droniou:22}; its lowest-order version was proved to be locking-free. Locking-free HDG schemes on polygonal meshes
have also been developed in \cite{Chen.Zhang.ea:25}. These contributions show that the combination of geometric flexibility and robustness in the thin-plate limit is a central objective in the design of modern plate discretisations. 

In this work, we propose and analyse the first scheme based on a BGG construction for the Reissner--Mindlin plate problem. At the continuous level, BGG techniques provide differential complexes relevant to elasticity and plate models; see, for instance, \cite{Arnold.Hu:21,Cap.Hu:24}. Within the finite element framework, BGG constructions have also been used to reinterpret existing stress and strain complexes; see \cite{Christiansen.Hu:23}. At the discrete level, the polygonal BGG diagram developed in \cite{Di-Pietro.Droniou.ea:26,Di-Pietro.Droniou.ea:26.1} combines a discrete Stokes complex with a tensorised DDR complex and produces the first twisted discrete complex on general polygonal meshes. The differential operators entering this complex naturally contain the shear and bending operators of the Reissner--Mindlin model.

Starting from this discrete twisted complex, we introduce an arbitrary-order scheme for the Reissner--Mindlin problem on polygonal meshes. The first formulation is obtained directly as the Hodge--Laplacian problem associated with the complex. In this formulation, the bending stabilisation acts on the discrete gradient of the rotation. The method has the expected consistency properties and the numerical experiments display the anticipated convergence rates. Its complete coercivity analysis would, however, require a discrete Korn inequality in which this gradient stabilisation controls the original rotation degrees of freedom. Such an estimate does not seem straightforward to establish.

This difficulty is closely related to the nonconforming character of the discrete spaces. In standard nonconforming Korn inequalities, the broken $H_1$-seminorm is generally bounded by the discrete symmetric gradient supplemented with terms directly controlling inter-element or cell-to-face jumps of the discrete rotation field; see \cite{Brenner:04}, \cite[Lemma~7.23]{Di-Pietro.Droniou:20}, and the VEM construction of \cite{Kwak.Park:22}. Motivated by this standard structure, we therefore consider a second formulation in which the consistent reconstructed symmetric-gradient term is unchanged, while the stabilisation is applied directly to the rotation degrees of freedom. This modification loses the strict Hodge--Laplacian interpretation, but it makes it possible to establish
a discrete Korn inequality adapted to the DDR spaces and, consequently, to
prove coercivity.

For this second formulation, we derive an arbitrary-order consistency and convergence estimate. At the lowest order, a suitable lifting argument yields an error estimate that is uniform with respect to the plate thickness, thereby showing that the scheme is locking-free. Numerical experiments are performed for both the Hodge--Laplacian formulation and the modified formulation, in order to compare their convergence properties and their behaviour in the thin-plate regime.
Additional numerical comparisons with the DDR--HHO scheme of \cite{Di-Pietro.Droniou:22} illustrate the benefits of the enhanced continuity provided by the BGG construction; see Section~\ref{sec.adventage.H2}.

The remainder of the paper is organised as follows. We first recall the continuous Reissner--Mindlin problem and its interpretation as the Hodge--Laplacian problem associated with a twisted complex in Section~\ref{sec:introduction}. In Section~\ref{sec:discrete.setting}, we introduce the discrete setting, spaces, operators, and scalar products. The discrete scheme and its variant are presented in Section~\ref{sec:discrete.schmes}. The stability and convergence analysis of the modified formulation, including the discrete Korn inequality and the lowest-order locking-free estimate, is presented in Sections~\ref{sec:analysis} and~\ref{sec:locking.free}. Numerical experiments are reported in Section~\ref{sec:numerical.results}. Finally, the convergence towards the limiting model is studied in Section~\ref{sec:conv.models}. A table summarising the notation is provided in Appendix~\ref{appendix:notations}.

\subsection{Continuous PDEs}

We consider the Reissner--Mindlin model for a clamped elastic plate occupying a bounded polygonal domain $\Omega \subset \mathbb{R}^2$. 
The unknowns are the transverse displacement $u:\Omega \to \mathbb{R}$ and the rotation vector 
$\bvec{\psi}:\Omega \to \mathbb{R}^2$ of the normal fibers. 
The shear stress vector $\bvec{\gamma}:\Omega \to \mathbb{R}^2$ is introduced as an auxiliary variable.
The strong formulation of the model is given in \eqref{eq:strong.rm}, where $t>0$ denotes the plate thickness, 
$\kappa>0$ is the shear correction factor, and $f$ represents the transverse load density.
Homogeneous clamped boundary conditions are imposed on $\partial\Omega$. 
\begin{subequations}\label{eq:strong.rm}
\begin{alignat}{2}
    \bvec{\gamma} + \text{div}(\bvec{C\GRAD\psi})&=0 &&\qquad \text{in} \ \Omega\\
     -\text{div}\bvec{\gamma} &=f &&\qquad \text{in} \ \Omega\\
     \bvec{\gamma} &=\frac{\kappa}{t^2}(\GRAD u -\bvec{\psi}) &&\qquad \text{in} \ \Omega\\
     \bvec{\psi}&=0, \ u=0 &&\qquad \text{on} \ \partial\Omega. \label{boundary.cond}
\end{alignat}
\end{subequations}
The fourth-order elasticity tensor $\bvec C$  is defined by $\bvec C \bvec A = \beta_0 \bvec A_s + \beta_1 \tr(\bvec A)\Id$ for every matrix $\bvec A$, where $\bvec A_s$ denotes the symmetric part of $\bvec A$ and $\beta_0,\beta_1>0$ are material-dependent constants; see, e.g., \cite[Section~1]{Di-Pietro.Droniou:22}.

Throughout the paper, regularity indices are systematically written as subscripts. Thus, $H_i(\Omega)$ denotes the usual Sobolev space while $H_{i,0}(\Omega)$ denotes the closure of $C_{\infty,c}(\Omega)$ in $H_i(\Omega)$. This convention will be used consistently for both continuous spaces and their discrete counterparts. Superscripts are reserved for Cartesian powers and polynomial degrees.

A weak formulation of the problem reads: find $(\bvec{\psi},u)\in \HSobz$ such that
\eqref{eq:weak.rm.cont} holds.
The bilinear form $A(\cdot,\cdot)$ is the sum of a bending contribution $a(\cdot,\cdot)$ and a shear contribution
$b(\cdot,\cdot)$, while the linear form $l(\cdot)$ accounts for the external loading.

\begin{equation}
\label{eq:weak.rm.cont}
    A((\bvec{\psi},u),(\bvec{\phi},v))=l(v), \ \forall (\bvec{\phi},v)\in \HSobz,
\end{equation}
where the bilinear form $A:\left[\HSob\right]^2 \longmapsto \Real$ and the linear form $l:H_1(\Omega)\longmapsto\Real$ are such that, for all $((\bvec{\xi},w),(\bvec{\phi},v))\in \left[\HSob\right]^2$,
\begin{equation*}
    A((\bvec{\xi},w),(\bvec{\phi},v)) := a(\bvec{\xi},\bvec{\phi})+b((\bvec{\xi},w),(\bvec{\phi},v)), \quad  l(v):=( f,v)_{L^2(\Omega)},
\end{equation*}
with
\begin{align*}
    a(\bvec{\xi},\bvec{\phi})&:= (\bvec{C \GRAD \xi,\GRAD\phi} )_{L^{2}(\Omega)^{2\times 2}},\\
    b((\bvec{\xi},w),(\bvec{\phi},v))&:=\frac{\kappa}{t^2}( \bvec{\xi}-\GRAD w ,\bvec{\phi} -\GRAD v)_{L^{2}(\Omega)^{2}}.
\end{align*}

\subsection{From twisted complex to Reissner--Mindlin equations}
\label{sec:Hodge-Laplace}

We will refer to the following complex as the twisted complex: 
\begin{equation}\label{twisted-stokes-HD}
    \begin{tikzcd}[ampersand replacement=\&, column sep=2em]
      0\arrow{r}\&
      \begin{pmatrix}
        H_{1,0}(\Omega)  \\
        H_{1,0}(\Omega)^{2}
      \end{pmatrix}
      \arrow{r}{
        \begin{pmatrix} \GRAD & -I \\ 0 & \GRAD \end{pmatrix}
      }\&[3.5em] \begin{pmatrix}
        \bvec{H}_{\ROT,0}(\Omega)  \\
     \boldsymbol    H_{\boldsymbol \ROT,0}(\Omega)
      \end{pmatrix} \arrow{r}{
        \begin{pmatrix} \ROT & -\sskw \\ 0 &\boldsymbol  \ROT \end{pmatrix}
      } \&[3.5em] \begin{pmatrix}
        L^{2}(\Omega)   \\
        L^{2}(\Omega)^{2}
      \end{pmatrix} \arrow{r}{} \&0,
    \end{tikzcd}
\end{equation}
 where, for sufficiently smooth vector- and tensor-valued functions
\[
\bvec{v} = \begin{pmatrix}
  v_1 \\ v_2
\end{pmatrix}: \Omega \to \Real^2\quad\text{ and }\quad
\bvec{\tau} = \begin{pmatrix}
  \tau_{11} & \tau_{12} \\ \tau_{21} & \tau_{22}
\end{pmatrix}: \Omega \to \Real^{2\times 2},
\]
$\ROT\bvec v\coloneq\partial_1v_2-\partial_2v_1$ is the scalar rotor, 
\[
\VROT\bvec\tau
\coloneq
\begin{pmatrix}
  \partial_1\tau_{12}-\partial_2\tau_{11}\\
  \partial_1\tau_{22}-\partial_2\tau_{21}
\end{pmatrix}
\]
is its row-wise extension to tensor fields, and $\sskw\bvec\tau\coloneq\tau_{12}-\tau_{21}$ the scalar skew-symmetrisation operator, $H_{\ROT}(\Omega)$ and $\bvec H_{\VROT}(\Omega)$ denote the spaces of square-integrable vector and tensor fields whose scalar and row-wise rotors respectively are also square-integrable, with the subscript $0$ indicating a vanishing tangential trace.

As explained in \cite{Cap.Hu:24}, the twisted complex is linked to the Reissner--Mindlin problem through the Hodge--Laplacian operator. Following the discussion in \cite{Di-Pietro.Droniou.ea:26.1}, we consider the following Hilbert complex and its adjoint operators,
\[
\begin{tikzcd}[ampersand replacement=\&, column sep=4em]
  X_{k-1}  \arrow[r, black, "D^{k-1}", shift=({0,0.3em})] \& X_k \arrow[r, black, "D^{k}", shift=({0,0.3em})] \arrow[l, black, "D^{\adj}_{k-1}", shift=({0,-0.3em})] \& X_{k+1} \arrow[l, black, "D^{\adj}_{k}",shift=({0,-0.3em})],
\end{tikzcd}
\]
where the adjoints are taken with respect to suitable $L^2$-like inner products, possibly involving positive-definite weights that account for physical parameters.
The Hodge--Laplacian on $X_k$ is defined by $\mathcal{L} \coloneqq D^{\adj}_{k}D^{k}+ D^{k-1}D^{\adj}_{k-1}$ and, for a source term $F\in (X_k)'$, the corresponding weak Hodge--Laplace problem reads: Find $u \in X_k$ such that
\begin{equation*}
  \langle D^{k}u , D^{k}q \rangle_{X_{k+1}} + \langle D^{\adj}_{k-1}u, D^{\adj}_{k-1}q\rangle_{X_{k-1}} = F(q)\qquad\forall q \in X_k.
\end{equation*}

The Hodge--Laplacian problem associated with the first space of the 
twisted complex \eqref{twisted-stokes-HD} reads as follows: find $
(\bvec\psi,u)\in \HSobz$ such that
\begin{equation*}
\left(
\mathcal D_0
\begin{pmatrix}
u\\
\bvec\psi
\end{pmatrix},
\mathcal D_0
\begin{pmatrix}
v\\
\bvec\phi
\end{pmatrix}
\right)_{ H_{\ROT}(\Omega) \times \bvec H_{\VROT}(\Omega)}=\left\langle F,
\begin{pmatrix}
v\\
\bvec\phi
\end{pmatrix}
\right\rangle,\qquad\forall (\bvec\phi,v)\in \HSobz,
\end{equation*}
where
\[
\mathcal D_0\coloneq
\begin{pmatrix}
\GRAD & -\Id\\
0 & \GRAD
\end{pmatrix},
\qquad\mathcal D_0
\begin{pmatrix}
u\\
\bvec\psi
\end{pmatrix}
=\begin{pmatrix}
\GRAD u-\bvec\psi\\
\GRAD\bvec\psi
\end{pmatrix}.
\]

Choosing a source term $F=(f,0)$ with $f\in L^2(\Omega)$ and endowing $H_{\ROT}(\Omega)  \times \bvec{H}_{\VROT}(\Omega)$ with the following suitable weighted  product,
\begin{align*}
    \left(
    \begin{pmatrix}
    \cdot\\
    \star
    \end{pmatrix} , 
     \begin{pmatrix}
    \cdot\\
    \star
    \end{pmatrix}\right)_{H_{\ROT}(\Omega)  \times \bvec{H}_{\VROT}(\Omega)}  = \frac{\kappa}{t^2}(\cdot,\cdot)_{L^2(\Omega)^2} + (\bvec{C} \star , \star)_{L^2(\Omega)^{2\times 2}} 
\end{align*}
we recover the Reissner--Mindlin weak formulation \eqref{eq:weak.rm.cont}. 

\section{Discrete setting}\label{sec:discrete.setting}

This section introduces the mesh notation and the polynomial spaces used
throughout the paper. 

\subsection{Mesh and notation}
\label{sec:setting}

Let $\Omega\subset\mathbb{R}^2$ be a bounded polygonal domain. We consider a
polygonal mesh $\mathcal{M}_h=(\Th,\Eh,\Vh)$ of $\Omega$. The set $\Th$ consists of a finite collection of pairwise disjoint open polygonal elements \(T\), with diameter \(h_T\), such that $\overline{\Omega}=\bigcup_{T\in\Th}\overline{T}.$ The meshsize is defined by $h\coloneq\max_{T\in\Th}h_T.$
The set $\Eh$ collects the open straight edges of the mesh, each edge
\(E\in\Eh\) having length \(h_E\), while \(\Vh\) denotes the set of mesh
vertices. The position vector of a vertex \(V\in\Vh\) is denoted by
\(\bvec{x}_V\).

We assume that the pair \((\Th,\Eh)\) satisfies the compatibility requirements
of \cite[Definition~1.4]{Di-Pietro.Droniou:20}. In particular, every mesh edge
belongs to the boundary of at least one element, and, for each \(T\in\Th\),
the boundary \(\partial T\) is the union of the closures of the edges collected
in the set \(\ET\). This setting allows, in particular, a straight portion of
an element boundary to be partitioned into several mesh edges, as may occur
after non-conforming local refinement.

For any mesh entity \(Y\in\Th\cup\Eh\), we denote by \(\mathcal{V}_Y\) the set
of its vertices. Each edge \(E\in\Eh\) is endowed with a fixed unit tangent
vector \(\tangent_E\), which determines its orientation. The associated unit
normal vector \(\normal_E\) is chosen so that
\((\tangent_E,\normal_E)\) is positively oriented.

For every $T\in\Th$ and $E\in\ET$, we introduce the orientation factor
$\omega_{TE}\in\{-1,1\}$ such that $\omega_{TE}\normal_E$ is the unit normal to $E$ pointing outwards from $T$. Similarly, for $E\in\Eh$ and $V\in\mathcal{V}_E$, the factor $\omega_{EV}\in\{-1,1\}$ is chosen so that
$\omega_{EV}\tangent_E$ points towards $V$.

Moreover, for $\bullet \in \{\Eh,\Vh\}$, we denote by $\bullet^{\rm b}$ the subset of mesh entities lying on the boundary, and we set $\bullet^{\rm i} \coloneq \bullet \setminus \bullet^{\rm b}$.

We consider a regular sequence of meshes in the sense of
\cite[Assumption~7.6]{Di-Pietro.Droniou:20}. Throughout the paper, the notation
\[
a\lesssim b
\]
means that \(a\leq Cb\), where the hidden constant \(C>0\) is independent of
the meshsize \(h\) and of the plate thickness. It may depend on the domain,
the mesh regularity parameter, the other material coefficients, and, when
polynomial spaces are involved, the polynomial degree. We write
\(a\simeq b\) whenever both \(a\lesssim b\) and \(b\lesssim a\) hold.

\subsection{Polynomial spaces}
\label{sec:polynomial.spaces}

For an integer \(\ell\) and a mesh entity \(Y\in\Th\cup\Eh\), we denote by
\(\Poly{\ell}(Y)\) the space formed by the restrictions to \(Y\) of
two-variable polynomials of total degree at most \(\ell\). We adopt the
convention $\Poly{\ell}(Y)\coloneq\{0\}$ for all $\ell\leq -1$.

The \(L^2(Y)\)-orthogonal projector onto \(\Poly{\ell}(Y)\) is denoted by $\lproj{\ell}{Y}:L^2(Y)\longrightarrow\Poly{\ell}(Y).$ Its component-wise extensions to vector- and tensor-valued functions are both
denoted by $\vlproj{\ell}{Y}$. The nature of the argument makes clear whether the vector- or tensor-valued
version is being used.

For \(\bullet\in\{T,h\}\), we define
\(\Poly{\ell}_{\mathrm c}(\mathcal{E}_\bullet)\) as the space of functions that
are continuous on $\bigcup_{E\in\mathcal{E}_\bullet}\overline{E}$ and whose restriction to every edge
\(E\in\mathcal{E}_\bullet\) belongs to \(\Poly{\ell}(E)\). In contrast, the
broken polynomial space on the mesh is denoted by
\[
\Poly{\ell}(\Th)\coloneq\left\{q\in L^2(\Omega)\,:\,q_{|T}\in\Poly{\ell}(T)\quad\forall T\in\Th\right\}.
\]
For a sufficiently smooth scalar-valued function $q$, we define its rotated
gradient by
\[
\CURL q\coloneq
\begin{pmatrix}
\partial_2 q\\
-\partial_1 q
\end{pmatrix}.
\]

For every element \(T\in\Th\), we select a point \(\bvec{x}_T\in T\) such that
\(T\) contains a ball centred at \(\bvec{x}_T\) whose diameter is $\simeq h_T$. For every integer \(\ell\geq0\), we then introduce the polynomial subspaces
\[
\Roly{\ell}(T) \coloneq \CURL\Poly{\ell+1}(T), \qquad \cRoly{\ell}(T) \coloneq (\bvec{x}-\bvec{x}_T)\Poly{\ell-1}(T).
\]
They yield the following direct, although non-orthogonal, decomposition of $\Poly{\ell}(T)^2$ \cite{Arnold:18}:
\begin{equation*}
\Poly{\ell}(T)^2=\Roly{\ell}(T)\oplus\cRoly{\ell}(T).
\end{equation*}
Finally, the \(L^2\)-orthogonal projectors onto $\Roly{\ell}(T)$ and $\cRoly{\ell}(T)$ are respectively denoted by $\Rproj{\ell}{T}$ and $\Rcproj{\ell}{T}$. The same notation is used for their component-wise extensions to the corresponding tensor-valued spaces, with the intended projector being unambiguously determined by the nature of the argument. 

\subsection{Discrete spaces and operators}
\label{sec:discrete.spaces.operators}

The following discrete complex, which provides a discrete counterpart of a higher regularity version of \eqref{twisted-stokes-HD}, is constructed using the BGG machinery described in \cite{Di-Pietro.Droniou.ea:26}. More precisely, in this construction, the roles of $H_1(\Omega)$ and $\bvec H_{\ROT}(\Omega)$ in the first row of \eqref{twisted-stokes-HD} are played by discrete spaces of $H_2(\Omega)$ and $H_1(\Omega)^2$, respectively. We briefly recall below the discrete spaces and operators involved in the first part of this complex, which is the part relevant to the construction of the Reissner--Mindlin scheme.

\begin{equation}\label{twisted-elasticity-boundary}
    \begin{tikzcd}[ampersand replacement=\&]
0\arrow{r}\&
      \begin{pmatrix}
        \XSgrado{\Th} \\
        \tXdRgrado{\Th}
      \end{pmatrix}
      \arrow{r}{
        \begin{pmatrix} \SGRAD{h} & -\Id \\ 0 &\tGRAD{h}  \end{pmatrix}
      }\&[3em]
      \begin{pmatrix}
        \tXdRgrado{\Th} \\
        \tXdRroto{\Th}
      \end{pmatrix}
    \end{tikzcd}
\end{equation} 

The use of a discrete counterpart of $H_2(\Omega)$ for the transverse displacement is motivated by both regularity and asymptotic considerations. For a convex domain $\Omega$ the transverse displacement solving \eqref{eq:weak.rm.cont} is expected to belong to $H_2(\Omega)$. This additional regularity is not required by the Reissner--Mindlin variational formulation, whose natural energy space for the displacement remains $H_{1}(\Omega)$, but it can be exploited at the discrete level. This choice is also natural from the physical viewpoint when considering the thin-plate limit. Indeed, owing to the factor $t^{-2}$ in the shear energy, a family of solutions with bounded energy is expected to formally satisfy, in $L^2$-norm,
\[
    \GRAD u-\bvec{\psi}\longrightarrow 0 \text{ as }t\longrightarrow 0.
\]
The limiting Kirchhoff constraint is therefore $\bvec{\psi}=\GRAD u$. Substituting this relation into the bending strain
gives
\[
    (\bvec{\GRAD}_s\bvec{\psi}) = (\bvec{\GRAD}_s(\GRAD u)) = \hess u,
\]
where $\hess \coloneqq \bvec{\GRAD} \GRAD$ is the Hessian operator.
Thus, in the thin-plate limit, the Reissner--Mindlin model formally reduces to the Kirchhoff--Love plate model, whose energy involves the Hessian of the transverse displacement. Requiring $u\in H_2(\Omega)$  is consequently consistent with the natural energy space of the limiting model and ensures that $\GRAD u\in H_1(\Omega)^2$. This relation between the Reissner--Mindlin and Kirchhoff--Love models is classical; see, for instance, \cite{Arnold.Falk:89}. The use of an $H_2$-conforming space for the transverse displacement also appears in the shear--deflection formulation of \cite{Beirao-da-Veiga.Mora.ea:19*1}.
The use of a more regular discrete space may improve the numerical behaviour of the approximations; see Section~\ref{sec.adventage.H2}.

\begin{remark}[Serendipity reduction]
     In the present work, we consider the non-serendipity version of this complex; that is, we do not employ the serendipity reduction introduced in this reference to decrease the polynomial degree of the cell unknowns. This choice is motivated by several considerations. First, it improves the legibility of the presentation, since the operators arising in the non-serendipity setting are simpler to define. Second, as only the first part of the twisted complex is required in the present application (see Section~\ref{sec:Hodge-Laplace}), there is no need to account for the discrete rotor operators appearing in the remainder of the complex. This is relevant because removing the serendipity reduction causes these discrete rotor operators to lose their local surjectivity.
\end{remark}

The first discrete space consists of discrete counterparts of $H_2(\Omega)$ and $H_1(\Omega)^2$, defined as follows:
\begin{align}
  \XSgrad{\Th}\coloneq \Big\{{}&\underline{q}_h=((q_T)_{T\in\Th},(q_E)_{E\in\Eh},(\Gqen)_{E\in\Eh},(q_V)_{V\in\Vh},(\Gqv)_{V\in\Vh})\,: \nonumber \\
         {}&q_T\in\Poly{k-1}(T)\quad\forall T\in\Th\,,\quad
         q_E\in\Poly{k-1}(E)\text{ and }\Gqen\in\Poly{k}(E)\quad\forall E\in\Eh\, \nonumber,\\
         {}&q_V\in\Real\text{ and }\Gqv\in\Real^2\quad\forall V\in\Vh\Big\}.
         \label{eq:def.Xhess}\\
           \tXdRgrad{\Th}\coloneq \Big\{{}&\underline{\bvec{v}}_h=((\bvec{v}_T)_{T\in\Th},(\bvec{v}_E)_{E\in\Eh},(\bvec{v}_V)_{V\in\Vh})\,:\nonumber\\
           {}&\bvec{v}_T\in\Poly{k}(T)^2\quad\forall T\in\Th\,,\quad
           \bvec{v}_E\in\Poly{k}(E)^2\quad\forall E\in\Eh\,, \nonumber\\
                {}&\bvec{v}_V\in\Real^2\quad\forall V\in\Vh\Big\}.
                \label{eq:def.Xgrad}
\end{align}

The second discrete space consists of discrete counterparts of $H_1(\Omega)^2$ and $\boldsymbol H_{\boldsymbol \ROT}(\Omega)$, the latter being 
\begin{align}
       \tXdRrot{\Th}\coloneq \Big\{{}&\underline{\bvec{\tau}}_h=((\bvec{\tau}_{\cvec{R},T},\bvec{\tau}_{\cvec{R},T}^{\compl})_{T\in\Th},(\bvec{\tau}_E)_{E\in\Eh})\,:\nonumber\\
        {}&(\bvec{\tau}_{\cvec{R},T},\bvec{\tau}_{\cvec{R},T}^{\compl})\in\Roly{k}(T)^2\oplus\cRoly{k+1}(T)^2\quad\forall T\in\Th\,,\quad
        \bvec{\tau}_E\in\Poly{k+1}(E)^2\quad\forall E\in\Eh\Big\}.
        \label{eq:def.Xrot}
\end{align}

To account for the boundary conditions \eqref{boundary.cond}, we define the following discrete counterparts of $H_2(\Omega)\cap H_{1,0}(\Omega)$, $H_{1,0}(\Omega)^2$ and $\bvec{H}_{\VROT,0}(\Omega)$, respectively:
\begin{align*}
  \XSgrado{\Th}
  &\coloneq
  \left\{
  \dofh{q}\in \XSgrad{\Th}
  \,:\,
  q_E=0 \quad \forall E\in\Eh^{\rm b},
  \quad
  q_V=0 \quad \forall V\in\Vh^{\rm b}
  \right\}, \\
  \tXdRgrado{\Th}
  &\coloneq
  \left\{
  \dofh{\bvec v} \in \tXdRgrad{\Th}
  \,:\,
  \bvec v_E=\bvec 0 \quad \forall E\in\Eh^{\rm b},
  \quad
  \bvec v_V=\bvec 0 \quad \forall V\in\Vh^{\rm b}
  \right\},\\
  \tXdRroto{\Th}
  &\coloneq
  \left\{
  \dofh{\bvec \xi} \in \tXdRrot{\Th}
  \,:\,
  \bvec\xi_E=\bvec 0 \quad \forall E\in\Eh^{\rm b}
  \right\}.
\end{align*}
Here, $\XSgrado{\Th}$ is the discrete counterpart of $H_2(\Omega)\cap H_{1,0}(\Omega)$, rather than $H_{2,0}(\Omega)$, since the normal-derivative degrees of freedom remain unconstrained.

We then introduce interpolators on these spaces to provide a representation of sufficiently smooth functions by vectors of polynomials. The interpolators $ \ISgrad:\Cspace{1}(\overline{\Omega})\to \XSgrad{\Th}$, $\IdRgrad:\Cspace{0}(\overline{\Omega})^2\to \tXdRgrad{\Th}$ and $\IdRrot:\Cspace{0}(\overline{\Omega})^{2\times 2}\to \tXdRrot{\Th}$ are defined by
\begin{alignat}{2}
    \ISgrad q\coloneq \Big({}&(\lproj{k-1}{T}q)_{T\in\Th},(\lproj{k-1}{E}q)_{E\in\Eh},(\lproj{k}{E}(\GRAD q\cdot\normal_E))_{E\in\Eh}, \nonumber
    \\
      {}&(q(\bvec{x}_V))_{V\in\Vh},(\GRAD q(\bvec{x}_V))_{V\in\Vh}
      \Big)&&\qquad\forall q\in \Cspace{1}(\overline{\Omega}), \label{eq:def.ISgrad}  \\
  \label{eq:def.IdRgrad}
  \IdRgrad \bvec{v}\coloneq {}&((\vlproj{k}{T}\bvec{v})_{T\in\Th},(\vlproj{k}{E}\bvec{v})_{E\in\Eh},(\bvec{v}(\bvec{x}_V))_{V\in\Vh})&&\qquad\forall \bvec{v}\in \Cspace{0}(\overline{\Omega})^2,\\
  \label{eq:def.IdRrot}
  \IdRrot \bvec{\tau}\coloneq {}&((\Rproj{k}{T} \bvec\tau,\Rcproj{k+1}{T}\bvec\tau)_{T\in\Th},(\vlproj{k+1}{E}(\bvec{\tau}\tangent_E))_{E\in\Eh})&&\qquad\forall \bvec{\tau}\in \Cspace{0}(\overline{\Omega})^{2\times 2}.
\end{alignat}

Throughout the paper, local versions of the discrete spaces are defined as follows. Given a mesh entity \(Y\in\Th\cup\Eh\), we retain all the unknowns associated with \(Y\), together with those attached to the lower-dimensional entities contained in its boundary. The resulting local space is denoted by replacing the global mesh argument \(\Th\) with \(Y\). For instance,
\[
\XSgrad{E}\coloneq
\begin{aligned}[t]
  \Big\{
  &\underline{q}_E
  = \big(q_E,\Gqen,(q_V)_{V\in\mathcal V_E},
  (\Gqv)_{V\in\mathcal V_E}\big)\,:
  \\
  &q_E\in\Poly{k-1}(E),\quad
  \Gqen\in\Poly{k}(E),\quad
  q_V\in\Real,\quad
  \Gqv\in\Real^2
  \quad\forall V\in\mathcal V_E
  \Big\}.
\end{aligned}
\]
Consistently with this convention, if \(\underline{v}_h\) is a global vector
of degrees of freedom, the collection of its components associated with
\(Y\) and its boundary is denoted by \(\underline{v}_Y\).

For $\dofh{\bvec v} \in \XdRgrad{\Th}$, we define its discrete gradient 
\begin{subequations}\label{eq:def.tGrad}
\begin{equation}\label{eq:def.uGh1}
\tGRAD{h}\dofh{\bvec v}\coloneq ((\Rproj{k}{T} \tGT\dof{\boldsymbol{v}}{T},\Rcproj{k+1}{T}\tGT\dof{\boldsymbol{v}}{T})_{T\in\Th},(\tGE\dof{\boldsymbol{v}}{E})_{E\in\Eh})
\in\tXdRrot{\Th},
\end{equation}
where, for all $E\in\Eh$ and all $T\in\Th$, the discrete edge gradient $\tGE\dof{\bvec v}{E}\in\Poly{k+1}(E)^2$ and discrete element gradient $\tGT\dofT{\bvec v}\in\Poly{k+1}(T)^{2\times 2}$ are respectively such that (with $\partial_{\tangent_E}$ the derivative along $E$ in the direction $\tangent_E$)
\begin{gather}
    \int_E \tGE\underline{\bvec{v}}_E\cdot\bvec{w}
    =-\int_E \bvec{v}_E\cdot \partial_{\tangent_E}\bvec{w}
    + \sum_{V\in\VE}\omega_{EV}\bvec{v}_V\cdot\bvec{w}(\bvec{x}_V)
    \qquad \forall \bvec{w}\in \Poly{k+1}(E)^2, \label{eq:def.tGrad.E}
    \\
    \int_T \tGT\underline{\bvec{v}}_T:\bvec{\zeta}
    = -\int_T \bvec{v}_T\cdot\DIV\bvec{\zeta}
    + \sum_{E\in\ET}\omega_{TE}\int_{E} \gammadR{\partial T}\dofT{\bvec{v}}\cdot (\bvec{\zeta} \normal_E)
    \qquad \forall\bvec{\zeta}\in\Poly{k+1}(T)^{2\times 2}.\nonumber
  \end{gather}
\end{subequations}
Above, the trace operator $\gammadR{\partial T}\dof{\bvec{v}}{T} \in\Poly{k+2}_c(\ET)^2$ is such that, for all $E\in\ET$, $\lproj{k}{E} \gammadR{\partial T}\dof{\bvec{v}}{T} = \bvec{v}_E$, and, for all $V\in\VT$, $\gammadR{\partial T}\dof{\bvec{v}}{T}(\bvec{x}_V) = \bvec{v}_V$. 
The discrete gradient of $\dofh{q}\in\XSgrad{\Th}$ is  given by
\begin{equation}\label{eq:def.nablah}
  \SGRAD{h}\dofh{q}=((\nablaT\underline{q}_T)_{T\in\Th},(\Gqet\dof{q}{E}\tangent_E + \Gqen \normal_E))_{E\in\Eh},(\Gqv)_{V\in\Vh})
  \in\XSrot{\Th},
\end{equation}
where, for all $E \in\Eh$ and all $T \in \Th$, the discrete tangential gradient $\Gqet \dof{q}{E} \in \Poly{k}(E)$ and the discrete element gradient $\nablaT\underline{q}_T\in\Poly{k}(T)^2$ are respectively such that
\begin{equation*}
  \int_E \Gqet\underline{q}_E\,r=-\int_E q_E \partial_{\tangent_E}r + \sum_{V\in\VE}\omega_{EV}\,q_V\,r(\bvec{x}_V)\qquad\forall r\in\Poly{k}(E).
\end{equation*}
\begin{equation*}
  \int_T\nablaT\underline{q}_T\cdot \bvec w =
  -\int_T q_T\DIV \bvec w
  + \sum_{E\in\ET}\omega_{TE}\int_E \gammaS \dofT{q}\, (\bvec w\cdot\normal_{E})\qquad
  \forall \bvec w\in\Poly{k}(T)^2.
\end{equation*}
The trace operator $\gammaS\dofT{q}$ is such that, for all $E\in\ET$,
$\lproj{k-1}{E}\,\gammaS\dof{q}{T} = q_E$, and, for all $V\in\VE$,
$\gammaS\dof{q}{T}(\bvec{x}_V) = q_V$.

On each cell $T\in\Th$, the discrete gradients are used to define the discrete potential operators: The discrete potential operator $\Pgrad{k+2}{T}:\tXdRgrad{T}\to\Poly{k+2}(T)^2$ is such that, for all $\dof{\bvec{v}}{T}\in\tXdRgrad{T}$,
\begin{equation}
\label{eq:def.potential.1}
 \int_T \Pgrad{k+2}{T} \dof{\bvec{v}}{T}\cdot \DIV \bvec{\chi}  = -\int_T \tGT  \dof{\bvec{v}}{T} : \bvec{\chi} + \sum_{E\in\ET} \omega_{TE} \int_E \gammadR{\partial T}\dof{\bvec{v}}{T} \cdot (\chi \normal_E) \qquad \forall \bvec \chi \in\cRoly{k+3}(T)^2.
\end{equation}
Then, the local potential operator $\SPgrad:\XSgrad{T}\to\Poly{k+1}(T)$ is such that:
\begin{equation}
\label{eq:def.potential.2}    
  \int_T \SPgrad \dof{q}{T}\DIV \bvec w
  = -\int_T \nablaTfull\dof{q}{T}\cdot \bvec w
    + \sum_{E\in\ET}\omega_{TE}\int_E \gammaS \dof{q}{T}(\bvec w \cdot \normal_E)
 \qquad \forall \bvec w\in \cRoly{k+2}(T).
\end{equation}

\begin{remark}[Rewriting the potential $\SPgrad$]
\label{rem:rewriting.pot}
In \cite[Eq.~(7.2)]{Di-Pietro.Droniou.ea:26}, for each $T\in\Th$, the definition of the potential \eqref{eq:def.potential.2} involves the operator $\SProt\SGRAD{T}$ instead of $\nablaTfull$, where $\SProt$ is a discrete potential on $\tXdRgrad{T}$. As we are not including any serendipity reduction in the complexes above, it can be seen that these two operators coincide thanks to \cite[Prop.~7]{Di-Pietro.Droniou:23}.
\end{remark}

We define the local discrete symmetric gradient $\tGsTfull$ as the symmetric part of the discrete gradient $\tGTfull$, that is, for all $\dofT{\bvec{v}} \in \tXdRgrad{T}$,
\[
\tGsTfull \dofT{\bvec{v}} \coloneq \frac{\tGTfull \dofT{\bvec{v}} + (\tGTfull \dofT{\bvec{v}})^{\top}}{2}.
\]
The symmetric gradient vector is such that
\[ 
\tGRADs{h} \dofh{\bvec v} \coloneq ((\Rproj{k}{T} \tGsTfull\dof{\boldsymbol{v}}{T},\Rcproj{k+1}{T}\tGsTfull\dof{\boldsymbol{v}}{T})_{T\in\Th},(\tGE\dof{\boldsymbol{v}}{E})_{E\in\Eh})
\in\tXdRrot{\Th}.
\]

Whenever an operator is defined locally on each mesh element \(T\in\Th\), we use the subscript \(h\) to denote its global broken counterpart, whose restriction to every \(T\in\Th\) coincides with the corresponding local operator. More precisely, if \(\mathcal O_T\) is the local operator, then $\mathcal O_h$ is defined such that \[ \left(\mathcal O_h \underline{v}_h\right)_{|T} = \mathcal O_T \underline{v}_T \qquad\forall T\in\Th. \]

\subsection{Discrete scalar products}
\label{sec:scalar.product}

The discrete scalar products are defined through a consistent part consisting of an $L^2$ product of the potential reconstructions plus a consistent stabilisation term. The scalar products on these spaces are defined as follows. For all $T\in \Th$,
\begin{align}
    ( \cdot, \cdot )_{2,T} &{}\coloneq ( \SPgrad \cdot ,\SPgrad \cdot )_{L^2(T)} + s_{2,T}(\cdot ,\cdot), \nonumber\\
    ( \cdot, \cdot )_{1,T} &{}\coloneq ( \Pgrad{k+2}{T} \cdot ,\Pgrad{k+2}{T} \cdot )_{L^2(T)^2} + s_{1,T}(\cdot ,\cdot), \label{eq:def.sp.grad}\\ 
    ( \cdot, \cdot )_{\VROT,T} &{}\coloneq ( \Prot{k+1}{T} \cdot ,\Prot{k+1}{T} \cdot )_{L^2(T)^{2\times 2}} + s_{\VROT,T}(\cdot ,\cdot), \label{eq:def.sp.rot}
\end{align}
where the stabilisation terms are described in \cite[Section 3.7]{Di-Pietro.Droniou.ea:26.1} and \cite[Eq.~(4.18)]{Di-Pietro.Droniou:23} (in the latter reference, the stabilisation considered there should be tensorised and defined at order $k+1$ to be consistent with the present framework). For the sake of legibility, we recall the definition of $s_{1,T}$: for all $\dofT{v},\dofT{w} \in\tXdRgrad{T}$,
\begin{equation}
\label{eq:def.s1T}
   s_{1,T}(\dofT{v},\dofT{w})\coloneq \sum_{E\in\ET} h_E \int_E (\Pgrad{k+2}{T}\dofT{v}-\gammadR{\partial T}\dofT{v})\cdot (\Pgrad{k+2}{T}\dofT{w}-\gammadR{\partial T}\dofT{w}) 
\end{equation}

The corresponding global scalar products and stabilisation forms are obtained by summing their local contributions. Namely, for $\star\in\left\{2,1,\VROT\right\}$,
\[
(\cdot,\cdot)_{\star,h}\coloneq\sum_{T\in\Th}(\cdot,\cdot)_{\star,T},\qquad s_{\star,h}(\cdot,\cdot)\coloneq\sum_{T\in\Th}s_{\star,T}(\cdot,\cdot).
\]

\section{Discrete schemes}
\label{sec:discrete.schmes}

\subsection{Discrete scheme as a Hodge--Laplacian problem}
\label{sec:discrete.scheme.hl}

We introduce the weighted discrete bilinear form on $\tXdRrot{\Th}$ defined by
\begin{equation*}
    (\cdot,\cdot)_{\bvec C,\VROT,h}\coloneq(\bvec C\,\Prot{k+1}{h}\cdot,\Prot{k+1}{h}\cdot)_{L^2(\Omega)^{2\times 2}}+s_{\VROT,h}(\cdot,\cdot).
\end{equation*}

The Hodge--Laplacian weak problem associated with the boundary complex
\eqref{twisted-elasticity-boundary} yields the following discrete scheme for the continuous problem
\eqref{eq:weak.rm.cont}: find
$(\dofh{\bvec{\psi}},\dofh{u})\in \tXdRgrado{\Th}\times \XSgrado{\Th}$ such that
\begin{equation}
\label{eq:weak.rm}
A_h((\dofh{\bvec{\psi}},\dofh{u}),(\dofh{\bvec{\phi}},\dofh{v}))=l_h(\dofh{v})\qquad\forall(\dofh{\bvec{\phi}},\dofh{v})\in \tXdRgrado{\Th}\times \XSgrado{\Th}.
\end{equation}
The bilinear form
$A_h:\left[\tXdRgrad{\Th}\times \XSgrad{\Th}\right]^2\to\Real$
and the linear form $l_h:\XSgrad{\Th}\to\Real$ are defined, for all $ ((\dofh{\bvec{\xi}},\dofh{w}),(\dofh{\bvec{\phi}},\dofh{v}))
\in \left[\tXdRgrad{\Th}\times \XSgrad{\Th}\right]^2$, by
\begin{equation}\label{eq:def.Ah}
A_h((\dofh{\bvec{\xi}},\dofh{w}),(\dofh{\bvec{\phi}},\dofh{v}))
\coloneq
a_h(\dofh{\bvec{\xi}},\dofh{\bvec{\phi}})+b_h((\dofh{\bvec{\xi}},\dofh{w}),(\dofh{\bvec{\phi}},\dofh{v})),\qquad l_h(\dofh{v})\coloneq( f , \SPgradh \dofh{v})_{L^2(\Omega)} ,
\end{equation}
with
\begin{align}
\label{eq:ah.hl}
a_h(\dofh{\bvec{\xi}},\dofh{\bvec{\phi}})&\coloneq(\tGRADs{h} \dofh{\bvec\xi},\tGRADs{h} \dofh{\bvec \phi})_{\bvec C,\VROT,h},\\
\label{eq:bh.hl}
b_h((\dofh{\bvec\xi},\dofh{w}),(\dofh{\bvec\phi},\dofh{v})) &\coloneq\frac{\kappa}{t^2}(\dofh{\bvec\xi}-\SGRAD{h}\dofh{w},\dofh{\bvec\phi}-\SGRAD{h}\dofh{v})_{1,h}.
\end{align}

Numerical tests confirm the expected behaviour of this scheme, namely convergence of order
$k+1$; see Section~\ref{sec:numerical.results}. Moreover, its consistency can be proved by following the same arguments as those used for the modified formulation below. The difficulty is not related to consistency, but to coercivity.
Indeed, the coercivity analysis of \eqref{eq:weak.rm} would require a discrete Korn inequality adapted to the stabilisation appearing in $(\cdot,\cdot)_{\bvec C,\VROT,h}$. More precisely, one would need an estimate of the form
\begin{equation}
\label{eq:korn.hl.required}
\|\dofh{\bvec\xi}\|_{1,h}^2 \lesssim
\|\tGshfull\dofh{\bvec\xi}\|_{L^2(\Omega)^{2\times 2}}^2
+s_{\VROT,h}(\tGRADs{h}\dofh{\bvec\xi},\tGRADs{h}\dofh{\bvec\xi})\qquad \forall \dofh{\bvec\xi}\in\tXdRgrado{\Th}.
\end{equation}
Such an estimate remains an open question. For this reason, the fully analysed scheme below uses a
modified stabilisation, chosen so as to fit the discrete Korn inequality available in our setting.

This Korn inequality is the key ingredient in the proof of coercivity; see Lemma~\ref{lem:coercivity}. The issue is specific to the non-conforming nature of polytopal methods and
does not arise in the same way for conforming discretisations, where the continuous Korn inequality can be used directly. We also note that similar strategies, involving stabilisations acting on the degrees of freedom rather than on their discrete gradients, are used in other polytopal non-conforming methods; see the discussion at the beginning of Section \ref{sec:discrete.korn} for references. For the same reason, such methods also depart from formulations that admit an interpretation as the Hodge--Laplacian associated with a complex.

If a Korn inequality of the form \eqref{eq:korn.hl.required} were available, possibly under additional restrictions on the tensor $\bvec C$, then the complete stability and convergence analysis would carry over to the Hodge--Laplacian formulation. For instance, assuming $\bvec C = \bvec \Id$ easily gives the expected inequality. Such a simplification also appears in the classical Reissner--Mindlin literature; see, e.g., \cite[p.~1279]{Arnold.Falk:89}.  In the present work, however, we base the proof on the variant introduced in the next section.

\subsection{Variant of the bilinear form \texorpdfstring{$a_h$}{a\_h}}
\label{sec:variant.ah}
We consider a variant of the discrete scheme \eqref{eq:weak.rm} in which only the stabilisation component of the bending bilinear form is modified. The discrete spaces, the consistent bending contribution, the shear bilinear form $b_h$, and the right-hand side $l_h$ remain unchanged. More precisely, we set
\begin{equation}
\label{eq:ah.modified}
a_h(\dofh{\bvec\xi},\dofh{\bvec\phi})\coloneq\big(\bvec C\tGshfull\dofh{\bvec\xi},\tGshfull\dofh{\bvec\phi}
\big)_{L^2(\Omega)^{2\times2}}+\beta_0\sum_{T\in\Th}h_T^{-2}s_{1,T}(\dofT{\bvec\xi},\dofT{\bvec\phi}).
\end{equation}

For comparison, using the definition of the scalar product \eqref{eq:def.sp.rot}, the bending bilinear form associated with the Hodge--Laplacian formulation can be written as
\begin{equation}
\label{eq:ah.simplification}
\big(\tGRADs{h}\dofh{\bvec\xi},\tGRADs{h}\dofh{\bvec\phi} \big)_{\bvec C,\VROT,h} = \big( \bvec C\tGshfull\dofh{\bvec\xi}, \tGshfull\dofh{\bvec\phi}\big)_{L^2(\Omega)^{2\times2}} +s_{\VROT,h} \big(\tGRADs{h}\dofh{\bvec\xi},\tGRADs{h}\dofh{\bvec\phi}\big).
\end{equation}
Hence, \eqref{eq:ah.modified} and \eqref{eq:ah.simplification} have the same consistent reconstructed symmetric-gradient contribution and differ only in their stabilisation terms.

Accordingly, the only modification required in the consistency analysis concerns the stabilisation contribution in \eqref{eq:def.consistency.grad.s}. The consistency of the stabilisation in \eqref{eq:ah.simplification} follows from the tensorised order-$(k+1)$ version of \cite[Eq.~(6.9)]{Di-Pietro.Droniou:23}, whereas that of the stabilisation in \eqref{eq:ah.modified} follows from the corresponding tensorised estimate for $s_{1,h}$; see \cite[Eq.~(6.8)]{Di-Pietro.Droniou:23}. Both stabilisations therefore provide the required approximation order. Since the remaining bending contribution, the shear term, and the load term are identical in the two formulations, the remainder of the consistency analysis is unchanged. In particular, replacing the stabilisation does not affect the order of the consistency error.

The distinction between the two formulations is essential for the coercivity analysis. The stabilisation in \eqref{eq:ah.modified} acts directly on the discrete vector unknown and is precisely the term required to apply the discrete Korn inequality \eqref{eq:korn}. By contrast, the stabilisation in \eqref{eq:ah.simplification} acts on its discrete gradient, and no corresponding discrete Korn inequality is available for this stabilisation. The modification \eqref{eq:ah.modified} therefore provides the control required for the coercivity argument, losing the strict interpretation of the scheme as the Hodge--Laplacian problem induced by the complex \eqref{twisted-elasticity-boundary}.

\section{Analysis}
\label{sec:analysis}

We define the following weighted energy semi-norm on $\tXdRgrad{\Th}\times \XSgrad{\Th}$,
\begin{equation}\begin{aligned}
\label{eq:def.norm.energy}
    \norm{1\times2,h}{(\dofh{\bvec \phi}, \dofh{v})}^2\coloneq{}& \norm{L^2(\Omega)^{2\times 2}}{\bvec{C} \tGshfull \dofh{\bvec\phi}}^2+ \beta_0 \sum_{T\in\Th} h_T^{-2}s_{1,T}(\dofT{\bvec \phi},\dofT{\bvec \phi}) + \frac{\kappa}{t^2} \norm{1,h}{\dofh{\bvec \phi} - \SGRAD{h}\dofh{v}}^2 \\
    &+\mu\left(\norm{1,h}{\dofh{\bvec\phi}}^2 + \norm{1,h}{\SGRAD{h}\dofh{v}}^2\right),
\end{aligned}
\end{equation}
where $\mu\coloneq\min(\kappa,\beta_0)$. Its restriction to  $\tXdRgrado{\Th}\times\XSgrado{\Th}$ is a norm; see discussion in \cite[Section 4.1]{Di-Pietro.Droniou:22}.

\subsection{Discrete Korn inequality}
\label{sec:discrete.korn}

In nonconforming discretisations of elasticity, discrete Korn inequalities
are generally obtained by complementing the symmetric-gradient term with a
stabilisation acting directly on the discrete rotation of fibers instead of its (discrete) gradient. This approach is used, for instance, in piecewise $H_1$, HHO, VEM,
and HDG formulations; see
\cite{Brenner:04,Di-Pietro.Droniou:20,Kwak.Park:22,Qiu.Shen.Shi:18}.
Such stabilisation terms directly control the nonconformity of the discrete
rotation and the element-wise rigid-body modes. In contrast, a stabilisation acting only on the
degrees of freedom of the discrete gradient would require an additional
estimate showing that these gradient degrees of freedom control the nonconformity and the rigid-body components of the original unknown. Since this property does not follow directly from the standard discrete Korn inequalities, we follow here the usual strategy and introduce a stabilisation acting on the rotation degrees of freedom.

In order to prove the coercivity of the general scheme (Lemma~\ref{lem:coercivity}), we first establish the following discrete version of Korn's inequality for the DDR method.

\begin{proposition}[Discrete Korn inequality for DDR spaces]
\label{prop:korn}
For all $\dofh{\bvec{\phi}} \in \tXdRgrado{\Th}$, it holds
\begin{equation}
\label{eq:korn}
\norm{\VROT,h}{\tGRAD{h}\dofh{\bvec \phi}}^2 \lesssim \norm{L^2(\Omega)^{2\times 2}}{\tGshfull\dofh{\bvec \phi}}^2 + \sum_{T\in\Th}h^{-2}_Ts_{1,T}(\dofT{\bvec \phi},\dofT{\bvec \phi}).
\end{equation}
\end{proposition}

We split the proof of Proposition~\ref{prop:korn} into two steps. First, we show that, locally, the discrete gradient (resp.~symmetric gradient) is equivalent to the gradient (resp.~symmetric gradient) of the discrete potential, up to stabilisation; see Lemma~\ref{lem:equivalence.G.grad.P}. Then, we prove that these potential reconstructions satisfy a discrete Korn inequality; see Lemma~\ref{lem:korn.on.grad.P}. Proposition~\ref{prop:korn} is a direct consequence of Lemma~\ref{lem:equivalence.G.grad.P} and Lemma~\ref{lem:korn.on.grad.P}. 

\begin{lemma}[Equivalence between discrete gradients and gradients of reconstructed potentials]
\label{lem:equivalence.G.grad.P}
For all $T\in\Th$, and all $\dofT{\bvec{\phi}} \in \tXdRgrad{T}$, it holds
\begin{align}
\label{eq:gradient.equivalence}
    \norm{\VROT,T}{\tGRAD{T} \dofT{\bvec \phi}}^2 &{}\lesssim \norm{L^2(T)^{2\times 2}}{\GRAD \Pgrad{k+2}{T} \dofT{\bvec \phi}}^2 + h_T^{-2}s_{1,T}(\dofT{\bvec \phi},\dofT{\bvec \phi}), \\
\label{eq:gradient.sym.equivalence}
    \norm{L^2(T)^{2\times 2}}{\GRAD_s \Pgrad{k+2}{T} \dofT{\bvec \phi}}^2 + h_T^{-2} s_{1,T}(\dofT{\bvec \phi},\dofT{\bvec \phi}) &{}\lesssim \norm{L^2(T)^{2\times 2}}{ \tGsTfull \dofT{\bvec \phi}}^2 + h_T^{-2} s_{1,T}(\dofT{\bvec \phi},\dofT{\bvec \phi}).
\end{align}
\end{lemma}

\begin{proof}
We first prove \eqref{eq:gradient.equivalence}. 
By the norm equivalence \cite[Lemma~5]{Di-Pietro.Droniou:23} and the $1$-Lipschitz property of the $L^2$-projectors $\Rproj{k}{T}$ and $\Rcproj{k+1}{T}$, we have 
\begin{equation}\label{eq:norm.equiv.rot}
\norm{\VROT,T}{\tGRAD{T} \dofT{\bvec \phi}}^2\lesssim \norm{L^2(T)^{2\times 2}}{\tGTfull \dofT{\bvec \phi}}^2+\sum_{E\in\ET}h_T\norm{L^2(E)^2}{\tGE \dofT{\bvec{\phi}}}^2.
\end{equation}
We therefore have to bound each term in the right-hand side of this estimate.

By \cite[Remark~17]{Di-Pietro.Droniou:23}, relation \eqref{eq:def.potential.1} extends to test functions in $\Poly{k+1}(T)^{2\times 2}$. Integrating by parts then yields, for all $\bvec\xi\in\Poly{k+1}(T)^{2\times 2}$,
    \begin{align}
    \label{eq:int.by.part}
     \int_T \tGTfull  \dofT{\bvec \phi} : \bvec \xi = \int_T \GRAD \Pgrad{k+2}{T} \dofT{\bvec \phi} : \bvec \xi  + \sum_{E \in \ET} \omega_{TE} \int_E (\gammadR{\partial T} \dofT{\bvec \phi}-\Pgrad{k+2}{T} \dofT{\bvec \phi} )\cdot (\bvec \xi \normal_E).
    \end{align}
We then take $\bvec \xi = \tGTfull \dofT{\bvec \phi}$ and apply the Cauchy--Schwarz inequality together with a discrete trace inequality \cite[Lemma~1.32]{Di-Pietro.Droniou:20}. Squaring the resulting estimate gives
    \begin{align}
    \norm{L^2(T)^{2\times 2}}{\tGTfull \dofT{\bvec \phi}}^2 &{}\lesssim \norm{L^2(T)^{2\times 2}}{\GRAD\Pgrad{k+2}{T} \dofT{\bvec \phi}}^2 + \sum_{E\in\ET} h^{-1}_E \norm{L^2(E)^2}{\gammadR{\partial T} \dofT{\bvec \phi}-\Pgrad{k+2}{T} \dofT{\bvec \phi}}^2\nonumber\\
     &{}\lesssim  \norm{L^2(T)^{2\times 2}}{\GRAD \Pgrad{k+2}{T} \dofT{\bvec \phi}}^2 + h_T^{-2} s_{1,T}(\dofT{\bvec \phi},\dofT{\bvec \phi}),
     \label{eq:norm.GT.norm.nablaPT.1}
    \end{align}
The conclusion follows directly from the definition of the stabilisation $s_{1,T}$ \eqref{eq:def.s1T}.
Moreover, replacing $(\bvec{v}_E,(\bvec{v}_V)_{V\in\VE})$ in the right-hand side of \eqref{eq:def.tGrad.E} with $(\gammadR{\partial T}\dofT{\bvec{v}},(\gammadR{\partial T}\dofT{\bvec{v}}(\bvec{x}_V))_{V\in\VE})$ (which is valid by definition of $\gammadR{\partial T}$ and since $\partial_{\tangent_E}\bvec{w}\in \Poly{k}(E)$), we notice that $\tGE \dofT{\bvec{v}}=\partial_{\tangent_E}\gammadR{\partial T} \dofT{\bvec{v}}$. Hence,
    \begin{align}
        h_T\norm{L^2(E)^2}{\tGE \dofT{\bvec{\phi}}}^2 &{}= h_T\norm{L^2(E)^2}{\partial_{\tangent_E}\gammadR{\partial T} \dofT{\bvec{\phi}} }^2 \nonumber\\
        &{} \lesssim \norm{L^2(T)^{2\times 2}}{\GRAD \Pgrad{k+2}{T} \dofT{\bvec \phi}}^2 +h_T\norm{L^2(E)^2}{\partial_{\tangent_E}\gammadR{\partial T} \dofT{\bvec{\phi}} -(\GRAD \Pgrad{k+2}{T} \dofT{\bvec \phi})\tangent_E}^2 \nonumber\\
        &{} \lesssim \norm{L^2(T)^{2\times 2}}{\GRAD \Pgrad{k+2}{T} \dofT{\bvec \phi}}^2 +h_T^{-1}\norm{L^2(E)^2}{\gammadR{\partial T} \dofT{\bvec{\phi}}  - \Pgrad{k+2}{T} \dofT{\bvec \phi}}^2 \nonumber\\
        &{}  \lesssim \norm{L^2(T)^{2\times 2}}{\GRAD \Pgrad{k+2}{T} \dofT{\bvec \phi}}^2 + h_T^{-2} s_{1,T}(\dofT{\bvec \phi},\dofT{\bvec \phi}).
        \label{eq:norm.GT.norm.nablaPT.2}
    \end{align}
Here, we have used the triangle inequality together with a discrete trace inequality in the first inequality, followed by a discrete inverse inequality \cite[Lemma~1.28]{Di-Pietro.Droniou:20}. Plugging \eqref{eq:norm.GT.norm.nablaPT.1} and \eqref{eq:norm.GT.norm.nablaPT.2} into \eqref{eq:norm.equiv.rot} concludes the proof of \eqref{eq:gradient.equivalence}.

We now prove \eqref{eq:gradient.sym.equivalence}. By taking test functions $\bvec \xi_s \in \Poly{k+1}(T)^{2 \times 2}_s$ (the symmetric part of the space $\Poly{k+1}(T)^{2 \times 2}$)  in \eqref{eq:int.by.part}, we obtain the following relation between the discrete symmetric gradient and the symmetric part of the gradient of the discrete potential: for all $\bvec \xi_s \in \Poly{k+1}(T)^{2 \times 2}_s$,
\[
    \int_T \GRAD_s \Pgrad{k+2}{T} \dofT{\bvec \phi} : \bvec\xi_s= 
     \int_T \tGsTfull \dofT{\bvec \phi} : \bvec \xi_s   - \sum_{E \in \ET} \omega_{TE} \int_E (\gammadR{\partial T} \dofT{\bvec \phi}-\Pgrad{k+2}{T} \dofT{\bvec \phi} )\cdot (\bvec \xi_s \normal_E).
\]
The remainder of the proof follows by adapting the arguments of the previous part, taking $\bvec \xi_s=\GRAD_s\Pgrad{k+2}{T} \dofT{\bvec \phi}$ and using the same discrete trace and inverse inequalities.
\end{proof}

\begin{lemma}[Korn inequality for the discrete potential]
\label{lem:korn.on.grad.P}
For all $\dofh{\bvec \phi} \in \tXdRgrado{\Th}$, it holds
\begin{equation*}
    \norm{L^2(\Th)^{2\times 2}}{\GRAD \Pgrad{k+2}{h} \dofh{\bvec \phi}}^2 + \sum_{T\in\Th}h^{-2}_Ts_{1,T}(\dofT{\bvec \phi},\dofT{\bvec \phi})\lesssim \norm{L^2(\Th)^{2\times 2}}{\GRAD_s \Pgrad{k+2}{h} \dofh{\bvec \phi}}^2 + \sum_{T\in\Th}h^{-2}_Ts_{1,T}(\dofT{\bvec \phi},\dofT{\bvec \phi})
\end{equation*}
\end{lemma}

\begin{proof}
We generalise the proof of \cite[Theorem 5.7]{Droniou.Haidar.ea:25}, which establishes the inequality for $k=0$, to arbitrary polynomial degree. We first apply \cite[Lemma 7.23]{Di-Pietro.Droniou:20}, with $\ell=k+2$, to get
    \[
    \norm{L^2(\Th)^{2\times 2}}{\GRAD \Pgrad{k+2}{h} \dofh{\bvec \phi}}^2 \lesssim \norm{L^2(\Th)^{2\times 2}}{\GRAD_s \Pgrad{k+2}{h} \dofh{\bvec \phi}}^2 + \sum_{E \in \Ehi} h_E^{-1}\norm{L^2(E)^2}{\llbracket \Pgrad{k+2}{h} \dofh{\bvec \phi}  \rrbracket_{E}}^2,
    \]
where, for any interior edge $E \in \Ehi$ shared by two cells $T_1,T_2 \in \Th$, and for any broken function $\bvec v$ with
well-defined traces on $E$, we set
\[
\llbracket \bvec v \rrbracket_E
\coloneq
(\bvec v_{\vert T_1})_{\vert E}
-
(\bvec v_{\vert T_2})_{\vert E}.
\]
Here, $T_1$ and $T_2$ are the two cells neighbouring $E$ (the order of which is irrelevant).
We then bound the term $\sum_{E \in \Ehi} h_E^{-1}\norm{L^2(E)^2}{\llbracket \Pgrad{k+2}{h} \dofh{\bvec \phi} \rrbracket_{E}}^2$. For $E\in\Ehi$, again denoting by $T_1$ and $T_2$ the two cells sharing $E$ as an edge, we note that $(\gammadR{\partial T_1}\underline{\bvec{\phi}}_{T_1})_{|E}=(\gammadR{\partial T_2}\underline{\bvec{\phi}}_{T_2})_{|E}$ since these functions only depend on the degrees of freedom of $\dofT{\bvec \phi}$ on $E$ and $\VE$.
We can therefore write
    \begin{align*}
        \norm{L^2(E)^2}{\llbracket \Pgrad{k+2}{h} \dofh{\bvec \phi} \rrbracket_{E}}^2 &{}=  \norm{L^2(E)^2}{ (\Pgrad{k+2}{T_1} \dofh{\bvec \phi})_{|E} - (\Pgrad{k+2}{T_2} \dofh{\bvec \phi})_{|E}}^2 \\
        &{} \leq 2\norm{L^2(E)^2}{ (\Pgrad{k+2}{T_1} \dofh{\bvec \phi}- \gammadR{\partial T_1}\dofh{\bvec \phi})_{|E}}^2 + 2\norm{L^2(E)^2}{ (\Pgrad{k+2}{T_2} \dofh{\bvec \phi}- \gammadR{\partial T_2}\dofh{\bvec \phi})_{|E}}^2\\
        &{} \lesssim h_{T_1}^{-1} s_{1,T_1}(\underline{\bvec \phi}_{T_1},\underline{\bvec \phi}_{T_1}) + h_{T_2}^{-1}s_{1,T_2}(\underline{\bvec \phi}_{T_2},\underline{\bvec \phi}_{T_2}).
    \end{align*}
Summing over all interior edges completes the proof.
\end{proof}

\subsection{Coercivity}

\begin{lemma}[Discrete Poincaré--Korn inequality]
For all $\dofh{\bvec \phi} \in \tXdRgrado{\Th}$, it holds
\begin{equation}
\label{eq:pc-korn}
    \norm{1,h}{\dofh{\bvec \phi}}^2 \lesssim  \norm{L^2(\Omega)^{2\times 2}}{\tGshfull\dofh{\bvec  \phi}}^2 + \sum_{T\in\Th}h_T^{-2}s_{1,T}(\dofT{\bvec \phi},\dofT{\bvec \phi})
\end{equation}
\end{lemma}

\begin{proof}
This is a direct consequence of a tensorised version of the Poincaré inequality
\cite[Lemma~7]{Di-Pietro:24}, combined with the discrete Korn inequality
\eqref{eq:korn}.
\end{proof}

\begin{lemma}[Coercivity]
\label{lem:coercivity}
For all $(\dofh{\bvec{\phi}},\dofh{v})\in \tXdRgrado{\Th} \times \XSgrado{\Th}$, defining $A_h$ by \eqref{eq:def.Ah} with \eqref{eq:bh.hl} and \eqref{eq:ah.modified}, it holds
\begin{equation*}
    \norm{1\times 2,h}{(\dofh{\bvec \phi}, \dofh{v})}^2 \lesssim A_h((\dofh{\bvec \phi}, \dofh{v}),(\dofh{\bvec \phi}, \dofh{v}))
\end{equation*}
\end{lemma}

\begin{proof}
The proof is a straightforward adaptation of \cite[Lemma~3]{Di-Pietro.Droniou:22}, using
\eqref{eq:pc-korn} instead of \cite[Eq.~(28)]{Di-Pietro.Droniou:22}.
\end{proof}

We are now ready to state and prove the convergence estimate for the scheme \eqref{eq:weak.rm} with the bilinear forms \eqref{eq:def.Ah}, \eqref{eq:bh.hl}, \eqref{eq:ah.modified}.

\begin{theorem}[Convergence]
\label{th:convergence}
Let $(\bvec \psi, u) \in \HSobz$ and $(\dofh{\bvec{\psi}},\dofh{u})\in \tXdRgrado{\Th} \times \XSgrado{\Th}$ be the solutions to problems \eqref{eq:weak.rm.cont} and \eqref{eq:weak.rm} (with the definitions \eqref{eq:def.Ah}, \eqref{eq:bh.hl}, \eqref{eq:ah.modified}), respectively. Assuming the additional regularity $u\in \Cspace{1}(\overline{\Omega}) \cap H_{k+4}(\Th)$ and $\bvec{\psi} \in H_{k+3}(\Th)^2$, it holds
\begin{multline*}
        \norm{1\times2,h}{(\dofh{\bvec{\psi}}-\IdRgrad\bvec{\psi},\dofh{u}-\ISgrad u)} 
        \\
        \lesssim h^{k+1}\mu^{-1/2} \seminorm{H_{k+1}(\Th)^2}{\bvec \gamma}  
        + h^{k+2}\beta_0^{-1/2}\seminorm{H_{k+2}(\Th)^{2\times 2}}{\bvec C \GRAD \bvec \psi}
        + h^{k+3}\kappa^{-1/2}t\seminorm{H_{k+3}(\Th)^2}{\bvec \gamma}.
\end{multline*}
\end{theorem}

\begin{remark}[Result with less regularity]
Following Remark~\ref{rem:reduced.estimate}, one can assume weaker regularity on $(u,\bvec\psi)$ in the above theorem, namely, $u\in\Cspace{1}(\overline{\Omega})\cap H_{k+2}(\Th)$ and $\bvec{\psi}\in H_{k+2}(\Th)^2$, and modify the proof of Proposition~\ref{prop:consitency} to obtain the following right-hand side bound in Theorem~\ref{th:convergence}:
    \[
h^{k+1}\mu^{-1/2} \seminorm{H_{k+1}(\Th)^2}{\bvec \gamma} + h^{k+1}\beta_0^{-1/2}\seminorm{H_{k+1}(\Th)^{2\times 2}}{\bvec C \GRAD \bvec \psi}.
    \]
\end{remark}

\begin{proof}
The proof follows directly from the coercivity of the scheme, Lemma~\ref{lem:coercivity}, together with its consistency, Proposition~\ref{prop:consitency}, by applying the third Strang lemma \cite[Theorem~10]{Di-Pietro.Droniou:18}.
\end{proof}

\begin{proposition}[Consistency estimate]
\label{prop:consitency}
Assume that $(\bvec \psi, u) \in \HSobz$ is the solution to problem \eqref{eq:weak.rm.cont}, with the additional regularity $u\in \Cspace{1}(\overline{\Omega}) \cap H_{k+4}(\Th)$ and $\bvec{\psi} \in H_{k+3}(\Th)^2$. We define the consistency error $\const{\cdot}{\cdot}$ by
 \begin{align*}
     \const{\dofh{\bvec \phi}}{\dofh{v}} = l_h(\dofh{v}) - A_h((\IdRgrad \bvec \psi,\ISgrad u),(\dofh{\bvec \phi}, \dofh{v})).
 \end{align*}
 Then, for all non-zero pairs $(\dofh{\bvec{\phi}},\dofh{v})\in \tXdRgrado{\Th} \times \XSgrado{\Th}$,
\[ \frac{\const{\dofh{\bvec \phi}}{\dofh{v}}}{\norm{1\times 2,h}{(\dofh{\bvec \phi}, \dofh{v})}}  \lesssim h^{k+1}\mu^{-1/2} \seminorm{H_{k+1}(\Th)^2}{\bvec \gamma}  + h^{k+2}\beta_0^{-1/2}\seminorm{H_{k+2}(\Th)^{2\times2}}{\bvec C \GRAD \bvec \psi}+ h^{k+3}\kappa^{-1/2} t \seminorm{H_{k+3}(\Th)^2}{\bvec \gamma}.
   \]
\end{proposition}

\begin{proof}
Using the same techniques as in \cite[Theorem~4]{Di-Pietro.Droniou:22}, we rewrite the consistency error as
\begin{align*}
    \const{\dofh{\bvec \phi}}{\dofh{v}} = \mathcal{E}_{\GRAD,h}(\bvec \gamma, \dofh{v})  + \mathcal{I} + \mathcal{E}_{\GRAD_s,h}(\bvec C \GRAD_s \bvec \psi,\dofh{\bvec \phi}),
\end{align*}
where each term is defined and bounded below.

We first consider
\[
\mathcal{E}_{\GRAD,h}(\bvec \gamma, \dofh{v}) \coloneq -\int_{\Omega} \DIV\bvec \gamma \SPgradh\dofh{v} - \sum_{T\in\Th} \int_T \bvec{\gamma} \cdot  \Pgrad{k+2}{T}\SGRAD{T} \dofT{v}.
 \]
Using \eqref{eq:def.potential.2}, and following \cite[Remark~17]{Di-Pietro.Droniou:23} to extend this relation to test functions in the polynomial space $\Poly{k}(T)$, we have, for all $\bvec w \in \Poly{k}(T)^2$,
\[
\int_T \SPgrad\dofT{v} \DIV \bvec w + \int_T  \vlproj{k}{T}\,\Pgrad{k+2}{T}\, \SGRAD{T}\dofT{v}\cdot \bvec w - \sum_{E\in\ET} \omega_{TE} \int_E \gammaS \dofT{v}\, (\bvec w\cdot\normal_E)=0,
\]
where the term $\nablaT\dofT{v}$ in the definition has been replaced by its equivalent expression
$ \vlproj{k}{T}\,\Pgrad{k+2}{T}\, \SGRAD{T}\dofT{v}$; see
\cite[Eq.~(4.31)]{Di-Pietro.Droniou:23}.
Noting that $\vlproj{k}{T}$ can be removed since $\bvec{w}\in\Poly{k}(T)^2$ and adding this relation to $\mathcal{E}_{\GRAD,h}(\bvec \gamma, \dofh{v})$ yields
\begin{multline*}
\mathcal{E}_{\GRAD,h}(\bvec \gamma, \dofh{v})\\
=\sum_{T\in\Th} \left[ \int_T \SPgrad\dofT{v} \DIV \bvec (\bvec w-\bvec  \gamma) + \int_T  \Pgrad{k+2}{T}\SGRAD{T}\dofT{v}\cdot (\bvec w-\bvec \gamma) - \sum_{E\in\ET} \omega_{TE} \int_E \gammaS \dofT{v}\, (\bvec w- \bvec \gamma) \cdot\normal_E\right].
\end{multline*}
Notice that the additional quantity $\sum_{T\in\Th} \sum_{E\in\ET} \omega_{TE} \int_E \gammaS \dofT{v} \bvec \gamma \cdot\normal_E$ introduced above vanishes since $\bvec{\gamma}$ is continuous across the edges (by assumption on $(u,\bvec\psi)$) and $(\gammaS \dofT{v})_{|E}=0$ if $E\in\Eh^{\rm b}$, owing to $\dof{v}{h}\in\XSgrado{\Th}$.
We choose $\bvec w = \vlproj{k}{T} \bvec \gamma$ in the above relation. Integrating by parts in the first term and then applying the Cauchy--Schwarz inequality, together with the discrete trace inequality \cite[Lemma~1.32]{Di-Pietro.Droniou:20} and the approximation property of the $L^2$-orthogonal projector \cite[Theorem~1.45]{Di-Pietro.Droniou:20}, we obtain
\begin{align*}
 |\mathcal{E}_{\GRAD,h}(\bvec \gamma, \dofh{v})| \lesssim{}& \sum_{T\in\Th} h^{k+1}_T\seminorm{H_{k+1}(T)^2}{\bvec \gamma}\left(\norm{L^2(T)^2}{\GRAD \SPgrad\dofT{v}} + \norm{L^2(T)^{2}}{\Pgrad{k+2}{T}\SGRAD{T}\dofT{v}}\right)\\
 &{} + \sum_{T\in\Th}\sum_{E \in \ET} h^{k+1}_T \seminorm{H_{k+1}(T)^2}{\bvec \gamma} h_E^{-1/2}\norm{L^2(E)}{\SPgrad\dofT{v} - \gammaS \dofT{v} }  \\
\lesssim{}& h^{k+1}\mu^{-1/2}\seminorm{H_{k+1}(\Th)^2}{\bvec \gamma}  \norm{1\times 2,h}{(\dofh{\bvec \phi}, \dofh{v})}.
\end{align*}
The last estimate follows from the fact that \cite[Lemma~7]{Di-Pietro.Droniou:23} can be adapted to the space $\XSgrad{T}$, yielding, for all $\dofT{v}\in\XSgrad{T}$,
\[
\norm{L^2(T)^2}{\GRAD \SPgrad\dofT{v}}^2 + \sum_{E\in\ET}h_E^{-1}\norm{L^2(E)}{\SPgrad\dofT{v} - \gammaS \dofT{v}}^2 \lesssim \norm{1,T}{\SGRAD{T}\dofT{v}}^2,
\]
while \eqref{eq:def.sp.grad} and \eqref{eq:def.norm.energy} yield
$\norm{L^{2}(T)^2}{\Pgrad{k+2}{T}\SGRAD{T}\dofT{v}}\lesssim \norm{1,h}{\SGRAD{h}\dofh{v}}\lesssim \mu^{-1/2}\norm{1\times2,h}{(\dofh{\bvec\phi},\dofh{v})}.$
Next, we define
 \[
 \mathcal{I} \coloneq \sum_{T\in\Th} \left[\int_T (\bvec \gamma - \Pgrad{k+2}{T}\IdRgradT\bvec \gamma)\Pgrad{k+2}{T}(\SGRAD{T}\dofT{v} - \dofT{\bvec \phi}) + s_{1,T}(\IdRgradT \bvec \gamma , \dofT{\bvec \phi} - \SGRAD{T}\dofT{v})  \right].
 \]
First applying the Cauchy--Schwarz inequality, and then using the consistency of the potential \cite[Eq.~(6.2)]{Di-Pietro.Droniou:23} for the first term and the consistency of the stabilisation forms \cite[Eq.~(6.8)]{Di-Pietro.Droniou:23} for the second one, we obtain
 \begin{align*}
     \seminorm{}{\mathcal{I}} &\leq \sum_{T\in\Th} \norm{L^2(T)^2}{\bvec \gamma - \Pgrad{k+2}{T}\IdRgradT\bvec \gamma } \norm{L^2(T)}{\Pgrad{k+2}{T}(\SGRAD{T}\dofT{v} - \dofT{\bvec \phi})} \\
     &\quad+ s_{1,T}(\IdRgradT \bvec \gamma ,\IdRgradT \bvec \gamma)^{1/2}s_{1,T}(\dofT{\bvec \phi} - \SGRAD{T}\dofT{v},\dofT{\bvec \phi} - \SGRAD{T}\dofT{v})^{1/2} \\
     &\lesssim \sum_{T\in\Th}  h^{k+3}\seminorm{H_{k+3}(T)^2}{\bvec \gamma} \left(\norm{L^2(T)}{\Pgrad{k+2}{T}(\SGRAD{T}\dofT{v} - \dofT{\bvec \phi})}+s_{1,T}(\dofT{\bvec \phi} - \SGRAD{T}\dofT{v},\dofT{\bvec \phi} - \SGRAD{T}\dofT{v})^{1/2}\right) \\
     \overset{\eqref{eq:def.norm.energy}}&\lesssim h^{k+3}\seminorm{H_{k+3}(\Th)^2}{\bvec \gamma} \kappa^{-1/2} t \norm{1\times 2,h}{(\dofh{\bvec \phi}, \dofh{v})}.
 \end{align*}

Finally, we consider the term
\begin{align}
  \mathcal{E}_{\GRAD_s,h}(\bvec C \GRAD_s \bvec \psi,\dofh{\bvec \phi})\coloneq{}& \sum_{T\in\Th} \int_T \DIV(\bvec C \GRAD_s \bvec \psi) \cdot \Pgrad{k+2}{T}\dofT{\bvec \phi} - \int_{\Omega}\bvec C \tGshfull\IdRgrad \bvec \psi : \tGshfull \dofh{\bvec \phi} \nonumber\\
  &{}+ \beta_0 \sum_{T\in\Th}h_T^{-2}s_{1,T}(\IdRgradT{\bvec \psi},\dofT{\bvec \phi}), 
\label{eq:def.consistency.grad.s}
\end{align}
To bound this term, we first observe that $\bvec C\bvec A_s=\bvec C\bvec A$ for all $\bvec A\in\mathbb{R}^{2\times2}$. Moreover, the commuting property $\tGRAD{h}\IdRgrad\bvec\psi = \IdRrot\GRAD\bvec\psi$ (see \cite[Eq.~(3.38)]{Di-Pietro.Droniou:23}) allows us to apply directly the adjoint consistency result \cite[Theorem~9]{Di-Pietro.Droniou:23}, row-wise, with $\bvec v=\bvec C\GRAD\bvec\psi$. Combining this result with the consistency of the stabilisation term \cite[Eq.~(6.8)]{Di-Pietro.Droniou:23} yields the required bound. We therefore obtain
\begin{align*}
     \mathcal{E}_{\GRAD_s,h}(\bvec C \GRAD_s \bvec \psi,\dofh{\bvec \phi}) &{}\lesssim h^{k+2}(\beta_0 + \beta_1) \beta_0^{-1} \seminorm{H_{k+2}(\Th)^{2\times 2}}{\bvec C \GRAD \bvec \psi}\norm{1\times 2,h}{(\dofh{\bvec \phi}, \dofh{v})}.
\end{align*}
Combining the estimates of the three terms in the decomposition of the consistency error yields the desired estimate.
\end{proof}

\section{Locking-free error estimate}
\label{sec:locking.free}

The following theorem highlights the locking-free property of the scheme at the lowest order. 
A numerical method for the Reissner--Mindlin problem is said to be locking-free if the error estimate remains valid uniformly as the plate thickness $t$ tends to zero. 
This uniformity relies on the standard regularity estimates for the Reissner--Mindlin problem on convex domains, according to which $t|\boldsymbol{\gamma}|_{H^1(\Omega)^2}$ remains bounded as $t\to0$, while $\|f\|_{L^2(\Omega)}$, $|\boldsymbol{\psi}|_{H^2(\Omega)^2}$, and $\|\boldsymbol{\gamma}\|_{L^2(\Omega)^2}$ remain bounded; see \cite{Arnold.Falk:89}.
The estimate below is therefore locking-free, owing to the factor $t$ multiplying the term $\seminorm{H_1(\Th)^2}{\bvec \gamma}$.

\begin{theorem}[Lowest-order error estimate]
\label{th:low-order.estimate}
Set $k=0$ and let $(\bvec \psi, u) \in \HSobz$ and $(\dofh{\bvec{\psi}},\dofh{u})\in \underline{H}_{1,0}^{1}(\Th)^2 \times \underline{H}_{2,1,0}^{0}(\Th)$ be the solutions to problems \eqref{eq:weak.rm.cont} and \eqref{eq:weak.rm}, respectively. Assume that $\Omega$ is convex. Then, it holds
    \begin{multline*}
    \norm{1\times2,h}{(\dofh{\bvec{\psi}}-\underline{\bvec{I}}_{\ddr,h}^{1}\bvec{\psi},\dofh{u}-\underline{I}_{\stokes,h}^{0} u)} \\
    \lesssim h\left( (\beta_0 + \beta_1) \beta_0^{-1} \seminorm{H_{2}(\Th)^2}{\bvec \psi}  + \beta_0^{-1/2}\norm{L^2(\Omega)^2}{\bvec \gamma}+ \kappa^{-1/2}t\seminorm{H_{1}(\Th)^2}{\bvec \gamma}+ \mu^{-1/2}\norm{L^2(\Omega)}{f} \right).
    \end{multline*}
\end{theorem}

\begin{proof}[Proof of Theorem~\ref{th:low-order.estimate}]
    We use a suboptimal version of \cite[Theorem~9]{Di-Pietro.Droniou:23} and \cite[Eq.~(6.8)]{Di-Pietro.Droniou:23}; see Remark~\ref{rem:reduced.estimate} below. This yields
    \[
    \mathcal{E}_{\GRAD_s,h}(\bvec C \GRAD_s \bvec \psi,\dofh{\bvec \phi}) \lesssim h(\beta_0 + \beta_1) \beta_0^{-1} \seminorm{H_{2}}{\bvec \psi} \norm{1\times 2,h}{(\dofh{\bvec \phi}, \dofh{v})}.
    \]
    The rest of the proof is a straightforward adaptation of \cite[Theorem~6]{Di-Pietro.Droniou:22}. 
    Notice that the lifting operators used in the reference must be adapted to the present spaces; these liftings are made explicit in Section~\ref{sec:lifting}.
\end{proof}

\begin{remark}[Reduced degree of approximation in some estimates]
\label{rem:reduced.estimate}
    We observe that the estimates in \cite[Eq.~(6.8)]{Di-Pietro.Droniou:23} and \cite[Theorem~9]{Di-Pietro.Droniou:23}, among others, can be used in a suboptimal form for any $\ell \in \llbracket 0 , k \rrbracket$, rather than only with the maximal degree $k$. 
    This follows from the fact that these estimates rely on the approximation properties of the $L^2$-orthogonal projector $\lproj{k}{\cdot}$, for which the approximation degree can be reduced; see \cite[Theorem~1.45]{Di-Pietro.Droniou:20}. 
\end{remark}

\subsection{Lifting}
\label{sec:lifting}

The proof of Theorem~\ref{th:low-order.estimate} relies on suitable lifting operators defined on the discrete spaces. In this section, we adapt these liftings to the present setting. For each $T\in\Th$, we first define a local reduction operator $\injgrad{T}: \XSgrad{T} \to \underline{H}_{1}^{k}(T)$, where
\begin{equation*}
\begin{aligned}
    \underline{H}_{1}^{k}(T) \coloneq \Big\{{}&
    \underline{q}_T=
    \bigl(q_T,(q_E)_{E\in\ET},(q_V)_{V\in\VT}\bigr)\,:\\
    &q_T\in\Poly{k-1}(T),\quad
    q_E\in\Poly{k-1}(E)\quad\forall E\in\ET,\quad
    q_V\in\Real\quad\forall V\in\VT
    \Big\}.
\end{aligned}
\end{equation*}
For all $\dofT{q}\in\XSgrad{T}$, this operator is simply defined by discarding the components corresponding to derivatives:
\begin{equation*}
\injgrad{T}(\dofT{q})
=
\bigl(q_T,(q_E)_{E\in\ET},(q_V)_{V\in\VT}\bigr).
\end{equation*}

We then define the local lifting operator $R_{2,T}:\underline{H}_{2}^{0}(T)\to\Poly{1}(T)$ by
$R_{2,T}\dofT{q}=\widetilde{\injgrad{T}\dofT{q}},$ where $\widetilde{\cdot}$ denotes the lifting introduced at the beginning of
\cite[Section~4.4]{Di-Pietro.Droniou:22}, restricted to the element $T$. The results of \cite[Lemma~9]{Di-Pietro.Droniou:22} remain valid with this lifting, replacing $\underline{\bvec G}_T^0$ by $\underline{\bvec G}_{2,T}^0$ and the scalar product $\norm{\bvec{\Theta},T}{\cdot}$ by $\norm{1,T}{\cdot}$. Here, $\underline{\bvec G}_T^0$ and $\norm{\bvec{\Theta},T}{\cdot}$ denote, respectively, the discrete gradient and scalar product defined in \cite[Section~3.2.1 and Eq.~(24)]{Di-Pietro.Droniou:22}.

Indeed, for $\dofT{q} \in \underline{H}_2^0(T)$, since $\bvec G_{2,T}^0\dofT{q}$ and $(\Gqet\dofT{q})_{E\in\ET}$ do not depend on the components $(\Gqv)_{V\in\VT}$ and $(\Gqen)_{E\in\ET}$, we have
\begin{align*}
\bvec G_{T}^0 \injgrad{T} \dofT{q} &= \bvec G_{2,T}^{0} \dofT{q}, \\
G_E^0  \injgrad{E} \dof{q}{E} &= \Gqet \dof{q}{E}, \quad \forall E \in \ET,
\end{align*}
where $\bvec G_{T}^0 \injgrad{T} \dofT{q}$ and $( G_E^0 \injgrad{E} \dof{q}{E})_{E\in\ET}$ are the components of $\underline{\bvec G}_T^0 \injgrad{T} \dofT{q}$. Moreover, since $\underline{\bvec G}_{2,T}^0 \dofT{q}$ contains more degrees of freedom than $\underline{\bvec G}_T^0 \injgrad{T} \dofT{q}$, we also have
\[
\norm{\bvec{\Theta},T}{\underline{\bvec G}_T^0 \injgrad{T} \dofT{q}} \leq \norm{1,T}{\underline{\bvec G}_{2,T}^0 \dofT{q}},
\]
because both norms are equivalent to the sum of appropriately scaled componentwise norms; see \cite[Lemma~5]{Di-Pietro.Droniou:23}.

On $\underline{H}_{1}^{1}(T)^2$, we first consider the decomposition
$
\underline{\bvec G}_{2,T}^0 \underline{H}_{2}^{0}(T) \oplus \Bigl(\underline{\bvec G}_{2,T}^0 \underline{H}_{2}^{0}(T)\Bigr)^{\perp}.
$
For any $\dofT{\bvec \phi} \in \underline{H}_{1}^{1}(T)^2$, we write the unique decomposition
\[
\dofT{\bvec \phi} = \underline{\bvec G}_{2,T}^0 \dofT{q} + \dofT{\bvec \theta}, \quad \text{with } \dofT{q}\in\underline{H}_{2}^{0}(T) \text{ and } \dofT{\bvec \theta} \perp \underline{\bvec G}_{2,T}^0 \underline{H}_{2}^{0}(T).
\] 
We then define the local lifting operator
$R_{1,T} : \underline H_1^1(T)^2 \to L^2(T)^2$ by
\[R_{1,T}\dofT{\bvec\phi}=
\GRAD R_{2,T}\dofT{q}+\lproj{0}{T}\Pgrad{2}{T}\dofT{\bvec\theta}.\]
With this choice of lifting, the results of \cite[Lemma~10]{Di-Pietro.Droniou:22} remain valid, and the proof carries over without any significant modification.

\section{Numerical results}
\label{sec:numerical.results}

We assess the numerical behaviour of the scheme \eqref{eq:weak.rm} with both choices \eqref{eq:ah.hl} and \eqref{eq:ah.modified}. All the computations are performed on the unit square $\Omega=(0,1)^2$. Unless stated otherwise, the Young's modulus and the Poisson ratio hidden in the physical parameters $\beta_0$ and $\beta_1$ are set to $E=1\text{ and } \nu=0.3.$

Let $(\bvec\psi,u)$ denote the solution of the continuous problem \eqref{eq:weak.rm.cont} and let $(\dofh{\bvec\psi},\dofh u)$ denote the solution of the discrete problem \eqref{eq:weak.rm}. The relative energy error is defined by
\begin{equation}
\label{eq:def.relative.energy.error}
E_h \coloneq \frac{\left\|\left(\dofh{\bvec\psi}-\IdRgrad\bvec\psi,\dofh u-\ISgrad u\right)\right\|_{1\times 2,h}}{\left\|\left(\IdRgrad\bvec\psi,\ISgrad u\right)\right\|_{1 \times 2,h}}.
\end{equation}

The first solution is a smooth polynomial test case used to assess the convergence rate with respect to the meshsize. The second solution exhibits a stronger dependence of the shear strain on the plate thickness and is therefore used to investigate the robustness of the method in the thin-plate regime.

For both test cases, the numerical tests are performed using the two variants of the scheme  \eqref{eq:weak.rm}. In the figures, the results are labelled \emph{HL} for the purely Hodge--Laplacian formulation, in which the bending bilinear form $a_h$ is defined by \eqref{eq:ah.hl}, and \emph{non-HL} for the modified formulation, in which the stabilisation acts directly on the primal degrees of freedom and $a_h$ is defined by \eqref{eq:ah.modified}. Note that, when computing the relative energy error \eqref{eq:def.relative.energy.error} for the \emph{HL} scheme, the bending stabilisation in the norm \eqref{eq:def.norm.energy} must be replaced by its \emph{HL} counterpart $s_{\VROT,h}$ (see Section~\ref{sec:scalar.product}).

\subsection{Polynomial solution}
\label{sec:test.polynomial}

The first test case is based on the exact polynomial solution used in \cite{Chinosi.Lovadina.Marini:06}. The transverse displacement is
defined, for all $\bvec{x}=(x_1,x_2)\in\Omega$, by
\begin{align*}
u(\bvec{x}) ={}& \frac{1}{3} x_1^3(1-x_1^3)x_2^3(1-x_2)^3\\
&-\frac{2t^2}{5(1-\nu)}  \Big[ x_2^3(x_2-1)^3x_1(x_1-1)(5x_1^2-5x_1+1) + x_1^3(x_1-1)^3x_2(x_2-1) (5x_2^2-5x_2+1) \Big],
\end{align*}
while the rotation field is given by
\begin{equation*}
\bvec{\psi}(\bvec{x}) =
\begin{pmatrix}
x_2^3(x_2-1)^3
x_1^2(x_1-1)^2(2x_1-1)
\\[0.2cm]
x_1^3(x_1-1)^3
x_2^2(x_2-1)^2(3x_2-1)
\end{pmatrix}.
\end{equation*}
The corresponding shear strain and transverse load are obtained from the
strong equations as
\[
\bvec{\gamma} = \frac{\kappa}{t^2} \bigl(\GRAD u-\bvec{\psi}\bigr), \qquad f=-\operatorname{div}\bvec{\gamma}.
\]

Both $u$ and $\bvec{\psi}$ satisfy the homogeneous clamped boundary conditions. Moreover, although the displacement depends explicitly on $t$, the corresponding shear strain is independent of the plate thickness.
This test case therefore provides a smooth setting in which the spatial convergence rates can be assessed without the influence of thickness-dependent boundary layers. In this test, we only verify the expected behaviour of the scheme, with a fixed thickness $t=1$. The expected convergence rate of order $k+1$, as established in Theorem~\ref{th:convergence}, is observed. Results are displayed in Figure \ref{fig:conv.polynomial}. 

 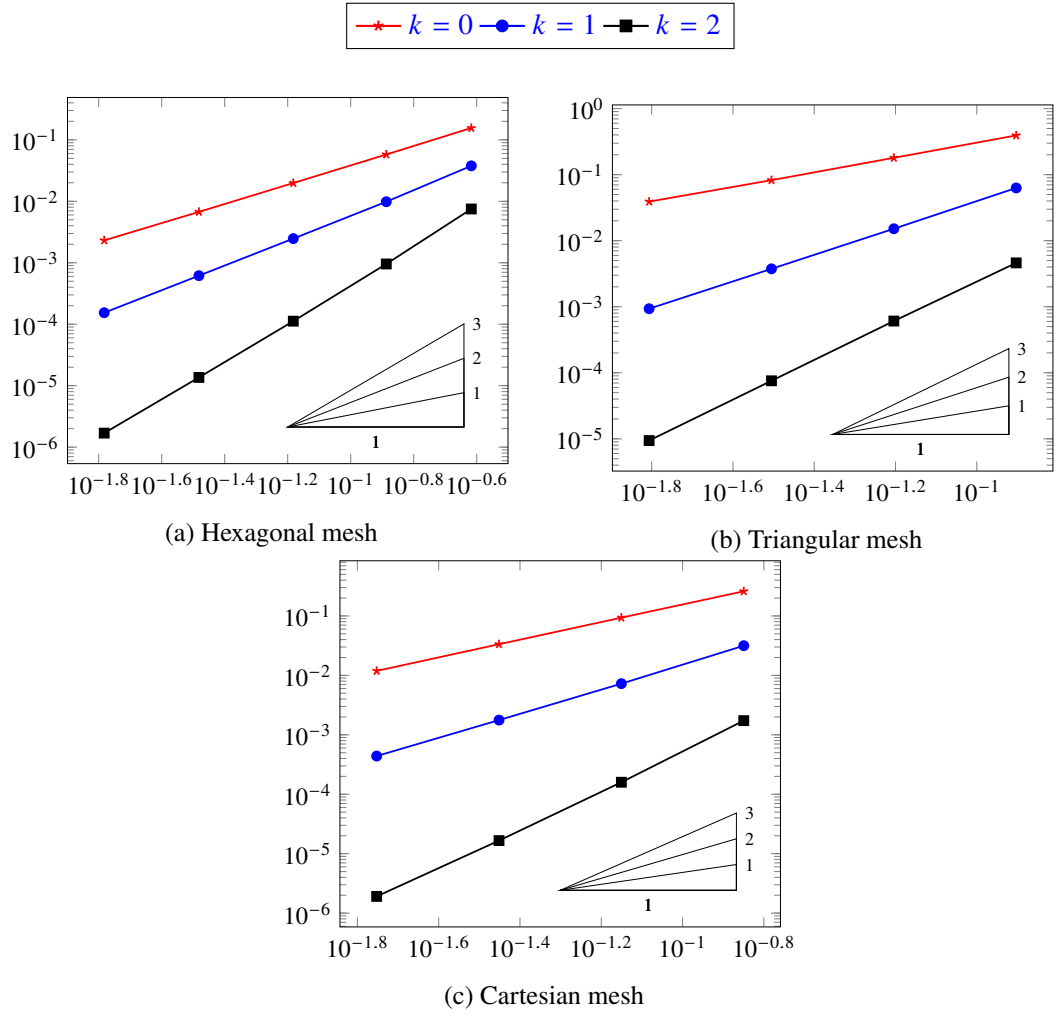
\begin{figure}\centering
   \ref{conv.hexa}
   \vspace{0.50cm}\\
  \begin{minipage}{0.45\textwidth}
     \begin{tikzpicture}[scale=0.85]
       \begin{loglogaxis}[legend columns=3, legend to name=conv.hexa]
         \addplot [thick, mark=star, red] table[x=MeshSize,y=EnError] {outputs/polynomial_solution/hexa_k0_t1/data_rates.dat};
         \logLogSlopeTriangle{0.90}{0.4}{0.1}{1}{black};
          \addlegendentry{$k=0$}
         \addplot [thick, mark=*, blue] table[x=MeshSize,y=EnError] {outputs/polynomial_solution/hexa_k1_t1/data_rates.dat};
         \logLogSlopeTriangle{0.90}{0.4}{0.1}{2}{black};
          \addlegendentry{$k=1$}
         \addplot [thick, mark=square*, black] table[x=MeshSize,y=EnError] {outputs/polynomial_solution/hexa_k2_t1/data_rates.dat};
         \logLogSlopeTriangle{0.90}{0.4}{0.1}{3}{black};
          \addlegendentry{$k=2$}
        \end{loglogaxis}            
     \end{tikzpicture}
     \subcaption{Hexagonal mesh}
   \end{minipage}
     \begin{minipage}{0.45\textwidth}
     \begin{tikzpicture}[scale=0.85]
       \begin{loglogaxis}[legend columns=3, legend to name=conv.hexaB]
         \addplot [thick, mark=star, red] table[x=MeshSize,y=EnError] {outputs/polynomial_solution/tri_k0_t1/data_rates.dat};
         \logLogSlopeTriangle{0.90}{0.4}{0.1}{1}{black};
          \addlegendentry{$k=0$}
         \addplot [thick, mark=*, blue] table[x=MeshSize,y=EnError] {outputs/polynomial_solution/tri_k1_t1/data_rates.dat};
         \logLogSlopeTriangle{0.90}{0.4}{0.1}{2}{black};
          \addlegendentry{$k=1$}
         \addplot [thick, mark=square*, black] table[x=MeshSize,y=EnError] {outputs/polynomial_solution/tri_k2_t1/data_rates.dat};
         \logLogSlopeTriangle{0.90}{0.4}{0.1}{3}{black};
          \addlegendentry{$k=2$}
        \end{loglogaxis}            
     \end{tikzpicture}
     \subcaption{Triangular mesh}
   \end{minipage}
     \begin{minipage}{0.45\textwidth}
     \begin{tikzpicture}[scale=0.85]
       \begin{loglogaxis} [legend columns=3, legend to name=conv.hexaC]
         \addplot [thick, mark=star, red] table[x=MeshSize,y=EnError] {outputs/polynomial_solution/cart_k0_t1/data_rates.dat};
         \logLogSlopeTriangle{0.90}{0.4}{0.1}{1}{black};
          \addlegendentry{$k=0$}
         \addplot [thick, mark=*, blue] table[x=MeshSize,y=EnError] {outputs/polynomial_solution/cart_k1_t1/data_rates.dat};
         \logLogSlopeTriangle{0.90}{0.4}{0.1}{2}{black};
          \addlegendentry{$k=1$}
         \addplot [thick, mark=square*, black] table[x=MeshSize,y=EnError] {outputs/polynomial_solution/cart_k2_t1/data_rates.dat};
         \logLogSlopeTriangle{0.90}{0.4}{0.1}{3}{black};
          \addlegendentry{$k=2$}
        \end{loglogaxis}            
     \end{tikzpicture}
     \subcaption{Cartesian mesh}
   \end{minipage}
   \caption{Error $E_h$ (see \eqref{eq:def.relative.energy.error}) with respect to the meshsize $h$ for the polynomial solution of Section \ref{sec:test.polynomial}. \emph{Non-HL} scheme. \label{fig:conv.polynomial}}
 \end{figure}

 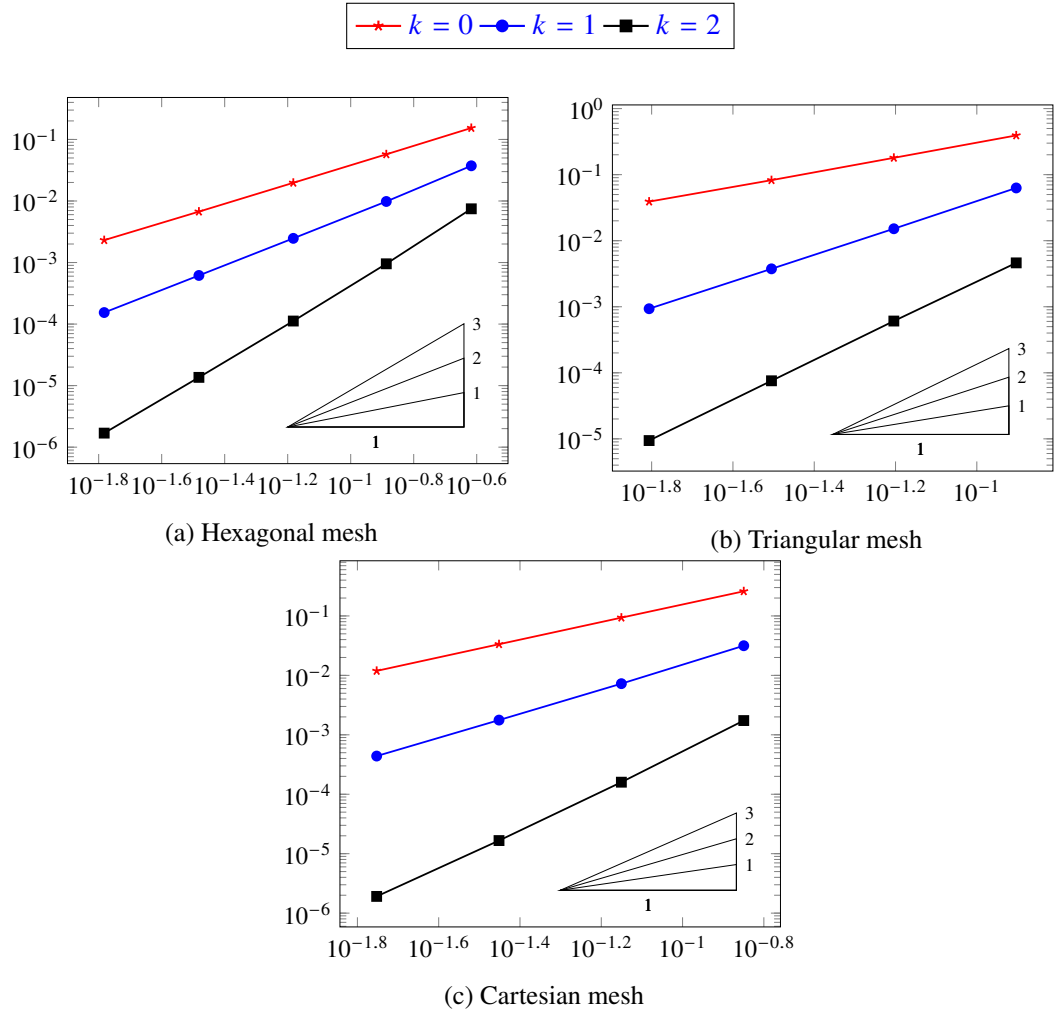
\begin{figure}\centering
   \ref{conv.hexa.2}
   \vspace{0.50cm}\\
  \begin{minipage}{0.45\textwidth}
     \begin{tikzpicture}[scale=0.85]
       \begin{loglogaxis}[legend columns=3, legend to name=conv.hexa.2]
         \addplot [thick, mark=star, red] table[x=MeshSize,y=EnError] {outputs/polynomial_solution_HL/hexa_k0_t1/data_rates.dat};
         \logLogSlopeTriangle{0.90}{0.4}{0.1}{1}{black};
          \addlegendentry{$k=0$}
         \addplot [thick, mark=*, blue] table[x=MeshSize,y=EnError] {outputs/polynomial_solution_HL/hexa_k1_t1/data_rates.dat};
         \logLogSlopeTriangle{0.90}{0.4}{0.1}{2}{black};
          \addlegendentry{$k=1$}
         \addplot [thick, mark=square*, black] table[x=MeshSize,y=EnError] {outputs/polynomial_solution_HL/hexa_k2_t1/data_rates.dat};
         \logLogSlopeTriangle{0.90}{0.4}{0.1}{3}{black};
          \addlegendentry{$k=2$}
        \end{loglogaxis}            
     \end{tikzpicture}
     \subcaption{Hexagonal mesh}
   \end{minipage}
     \begin{minipage}{0.45\textwidth}
     \begin{tikzpicture}[scale=0.85]
       \begin{loglogaxis}[legend columns=3, legend to name=conv.hexa.2B]
         \addplot [thick, mark=star, red] table[x=MeshSize,y=EnError] {outputs/polynomial_solution_HL/tri_k0_t1/data_rates.dat};
         \logLogSlopeTriangle{0.90}{0.4}{0.1}{1}{black};
          \addlegendentry{$k=0$}
         \addplot [thick, mark=*, blue] table[x=MeshSize,y=EnError] {outputs/polynomial_solution_HL/tri_k1_t1/data_rates.dat};
         \logLogSlopeTriangle{0.90}{0.4}{0.1}{2}{black};
          \addlegendentry{$k=1$}
         \addplot [thick, mark=square*, black] table[x=MeshSize,y=EnError] {outputs/polynomial_solution_HL/tri_k2_t1/data_rates.dat};
         \logLogSlopeTriangle{0.90}{0.4}{0.1}{3}{black};
          \addlegendentry{$k=2$}
        \end{loglogaxis}            
     \end{tikzpicture}
     \subcaption{Triangular mesh}
   \end{minipage}
     \begin{minipage}{0.45\textwidth}
     \begin{tikzpicture}[scale=0.85]
       \begin{loglogaxis} [legend columns=3, legend to name=conv.hexa.2C]
         \addplot [thick, mark=star, red] table[x=MeshSize,y=EnError] {outputs/polynomial_solution_HL/cart_k0_t1/data_rates.dat};
         \logLogSlopeTriangle{0.90}{0.4}{0.1}{1}{black};
          \addlegendentry{$k=0$}
         \addplot [thick, mark=*, blue] table[x=MeshSize,y=EnError] {outputs/polynomial_solution_HL/cart_k1_t1/data_rates.dat};
         \logLogSlopeTriangle{0.90}{0.4}{0.1}{2}{black};
          \addlegendentry{$k=1$}
         \addplot [thick, mark=square*, black] table[x=MeshSize,y=EnError] {outputs/polynomial_solution_HL/cart_k2_t1/data_rates.dat};
         \logLogSlopeTriangle{0.90}{0.4}{0.1}{3}{black};
          \addlegendentry{$k=2$}
        \end{loglogaxis}            
     \end{tikzpicture}
     \subcaption{Cartesian mesh}
   \end{minipage}
   \caption{Error $E_h$ (see \eqref{eq:def.relative.energy.error}) with respect to the meshsize $h$ for the polynomial solution of Section \ref{sec:test.polynomial}. \emph{HL} scheme. \label{fig:conv.polynomial.2}}
 \end{figure}

\subsection{Analytical solution}
\label{sec:test.analytical}

We now consider the analytical solution introduced in \cite[Section~5.2]{Di-Pietro.Droniou:22}, which is designed to provide a more representative behaviour in the thin-plate regime.

The transverse displacement and the rotation are defined by
\begin{equation*}
u(\bvec{x})=v(t,\bvec{x})+t^2w(t,\bvec{x}),\qquad\bvec{\psi}(\bvec{x})=\GRAD v(t,\bvec{x}),
\end{equation*}
where
\begin{alignat*}{2}
v(t,\bvec{x})&=t^3V(t^{-1}\bvec{x})+g(\bvec{x}),&\qquad w(t,\bvec{x})&=-\frac{\beta_0+\beta_1}{\kappa}\Delta v(t,\bvec{x}),\\
V(\bvec{y})&= y_1e^{-y_1}\cos(y_2),&\qquad g(\bvec{x})&= \sin(\pi x_1)\sin(\pi x_2).
\end{alignat*}

The corresponding shear strain and transverse load are given by
\begin{equation*}
\bvec{\gamma}(t,\bvec{x})=-(\beta_0+\beta_1)\GRAD\Delta v(t,\bvec{x}),
\end{equation*}
and
\begin{equation*}
f(\bvec{x})=(\beta_0+\beta_1)\Delta^2g(\bvec{x}).
\end{equation*}
In particular, the load is independent of the plate thickness. Since this solution does not satisfy homogeneous clamped boundary conditions, its exact boundary values are imposed in the discrete problem.

As $t\to0$, this solution satisfies
\begin{equation*}
\|u\|_{H_3(\Omega)}\sim 1,\qquad\|\bvec{\psi}\|_{H_2(\Omega)^2}\sim 1,\qquad\|\bvec{\gamma}\|_{L^2(\Omega)^2}\sim 1,
\end{equation*}
whereas, for every integer $s\geq1$,
\begin{equation*}
|\bvec{\gamma}|_{H_s(\Omega)^2}\sim t^{-s+\frac{1}{2}}.
\end{equation*}
Thus, the shear strain remains bounded in $L^2(\Omega)^2$, while its higher-order derivatives grow as the plate becomes thinner. This test case is therefore used to assess the robustness of the method in the thin-plate regime. We perform computations for several values of $t$ and at the lowest order $k=0$. For the sake of legibility, we only report the results obtained using the scheme \eqref{eq:weak.rm} with the choice \eqref{eq:ah.modified}. The results, presented in Figure~\ref{fig:conv.analytical}, confirm the robustness of the method with respect to $t$ and exhibit the first-order convergence rate established in Theorem~\ref{th:low-order.estimate}.
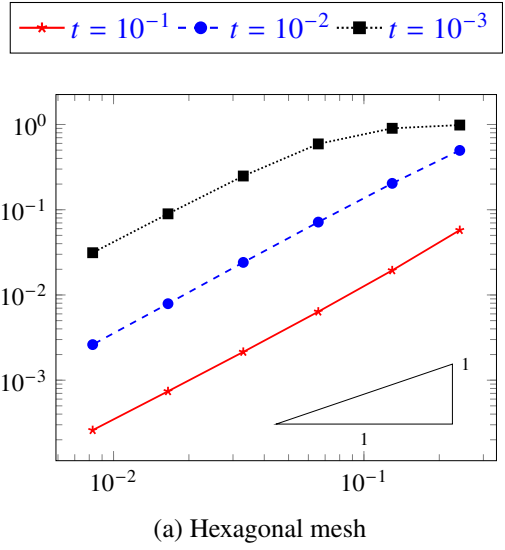
\begin{figure}\centering
  \ref{conv.hexa.analytical}
  \vspace{0.50cm}\\
  \begin{minipage}{0.45\textwidth}
    \begin{tikzpicture}[scale=0.85]
      \begin{loglogaxis}[legend columns=3, legend to name=conv.hexa.analytical]
        \addplot [thick, mark=star, mark options=solid, red] table[x=MeshSize,y=EnError] {outputs/lockingfree/hexa_k0_t1e-1/data_rates.dat};
        \addlegendentry{$t=10^{-1}$}
        \logLogSlopeTriangle{0.90}{0.4}{0.1}{1}{black};
        \addplot [thick, mark=*, mark options=solid, blue, dashed] table[x=MeshSize,y=EnError] {outputs/lockingfree/hexa_k0_t1e-2/data_rates.dat};
        \addlegendentry{$t=10^{-2}$}
        \addplot [thick, mark=square*, mark options=solid, black, densely dotted] table[x=MeshSize,y=EnError] {outputs/lockingfree/hexa_k0_t1e-3/data_rates.dat};
        \addlegendentry{$t=10^{-3}$}
      \end{loglogaxis}            
    \end{tikzpicture}
    \subcaption{Hexagonal mesh}
  \end{minipage}
  \caption{Error $E_h$ (see \eqref{eq:def.relative.energy.error}) with respect to the meshsize $h$ for the analytical solution of Section \ref{sec:test.analytical}, for $k=0$. \label{fig:conv.analytical}}
\end{figure}

\subsection{Numerical impact of the \texorpdfstring{$H_2$}{H2}-based construction}
\label{sec.adventage.H2}

We investigate whether the enhanced continuity of the $H^2$-based DDR--BGG construction leads to improved numerical results by comparing it with the DDR--HHO method for the Reissner--Mindlin problem introduced in \cite{Di-Pietro.Droniou:22}. The results presented here are valid for the scheme \eqref{eq:weak.rm} with both choices \eqref{eq:ah.hl} and \eqref{eq:ah.modified}. Moreover, at the lowest order $k=0$, the DDR--BGG energy error exhibits an observed convergence rate of approximately $1.5$ on both hexagonal and Cartesian meshes. This contrasts with the rate close to
one observed for the DDR--HHO scheme; see \cite[Figure~2]{Di-Pietro.Droniou:22}.

The effect of this enhanced continuity is expected to be reflected most directly in the interelement jumps of the reconstructed quantities. To measure these jumps consistently, for any broken vector-valued reconstruction $\bvec q_h$ and any interior edge $E\in\mathcal E_h^{\mathrm i}$, we introduce the local indicator
\begin{equation}
J_{1,E}(\bvec q_h)\coloneq h_E^{-1/2}\norm{L^2(E)^2}{[\bvec q_h]_E},
\label{eq.jump.discrete.gradient}
\end{equation}
and the corresponding global indicator
\begin{equation}
J_{1,h}(\bvec q_h)\coloneq\left(\sum_{E\in\mathcal E_h^{\mathrm i}}J_{1,E}(\bvec q_h)^2\right)^{1/2}
=\left(\sum_{E\in\mathcal E_h^{\mathrm i}}h_E^{-1}\norm{L^2(E)^2}{[\bvec q_h]_E}^2\right)^{1/2}.
\label{eq:def.jump.rotation.h1}
\end{equation}
The factor $h_E^{-1}$ gives the jump term the same scaling as the elementwise $L^2$-norm of the gradient. 

We denote by $P_h\dofh{\bvec{\psi}}$ the broken polynomial reconstruction of the discrete rotation, where $P_h$ denotes the corresponding broken potential reconstruction operator and $\dofh{\bvec{\psi}}$ the discrete approximation of the rotation for the method under consideration. Figure~\ref{fig:jump.theta.comparison} compares the global indicator $J_{1,h}(P_{h}\dofh{\bvec \psi})$ for the DDR--HHO and DDR--BGG schemes. For the analytical solution on triangular meshes with $k=1$, the jump converges with orders two and three for DDR--HHO and DDR--BGG, respectively. For the polynomial solution on hexagonal meshes with $k=0$, the corresponding observed orders are one and two. These results highlight the improved numerical behaviour provided by the enhanced continuity of the DDR--BGG construction.

\begin{figure}\centering
  \ref{conv.theta.comparison}
  \vspace{0.50cm}\\
  \begin{minipage}{0.45\textwidth}
    \begin{tikzpicture}[scale=0.85]
      \begin{loglogaxis}[legend columns=2, legend to name=conv.theta.comparison]
        \addplot [thick, mark=star, mark options=solid, red]
        table[x expr={sqrt(\thisrow{DimEXCurl}+\thisrow{DimXGrad})},y=JumpThetaH1]
        {outputs/comparison/rmddr/tri_s2_k1_t1/data_rates.dat};
        \addlegendentry{DDR--HHO}

        \addplot [thick, mark=*, mark options=solid, blue, dashed]
        table[x expr={sqrt(\thisrow{DimVSXGrad}+\thisrow{DimXHess})},y=JumpThetaH1]
        {outputs/comparison/rmbgg/tri_s2_k1_t1/data_rates.dat};
        \addlegendentry{DDR--BGG}
      \end{loglogaxis}
    \end{tikzpicture}
    \subcaption{Analytical solution, triangular mesh, $k=1$}
  \end{minipage}
  \hfill
  \begin{minipage}{0.45\textwidth}
    \begin{tikzpicture}[scale=0.85]
      \begin{loglogaxis}[legend columns=2, legend to name=conv.theta.comparisonB]
        \addplot [thick, mark=star, mark options=solid, red]
        table[x expr={sqrt(\thisrow{DimEXCurl}+\thisrow{DimXGrad})},y=JumpThetaH1]
        {outputs/comparison/rmddr/hexa_s1_k0_t1/data_rates.dat};
        \addlegendentry{DDR--HHO}

        \addplot [thick, mark=*, mark options=solid, blue, dashed]
        table[x expr={sqrt(\thisrow{DimVSXGrad}+\thisrow{DimXHess})},y=JumpThetaH1]
        {outputs/comparison/rmbgg/hexa_s1_k0_t1/data_rates.dat};
        \addlegendentry{DDR--BGG}
      \end{loglogaxis}
    \end{tikzpicture}
    \subcaption{Polynomial solution, hexagonal mesh, $k=0$}
  \end{minipage}

  \caption{Scaled jump of the reconstructed rotation $J_{1,h}(P_h\dofh{\bvec\psi})$ (see \eqref{eq:def.jump.rotation.h1}) as a function of the square root of the total number of DOFs.}
  \label{fig:jump.theta.comparison}
\end{figure}
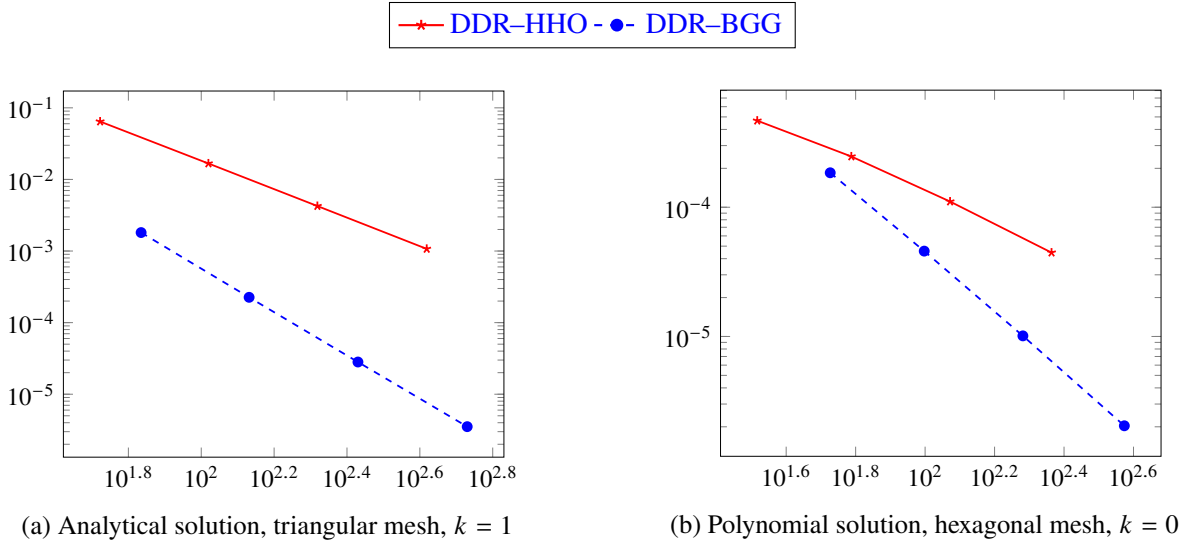

To complement this quantitative comparison with a qualitative illustration, we consider a test case with a discontinuous transverse load
\begin{equation*}
f(\bvec{x})=
\begin{cases}
100, & \bvec{x}\in B_{0.12}\bigl((0.35,0.5)\bigr),\\
-100, & \bvec{x}\in B_{0.12}\bigl((0.65,0.5)\bigr),\\
0, & \text{otherwise},
\end{cases}
\end{equation*} 
where $B_r(\bvec{x}_0)$ denotes the ball of radius $r$ centred at $\bvec{x}_0$.

The two localised loads of opposite signs generate a sharp transition in the quantities derived from the displacement near the centre of the plate. As no exact solution is available for this problem, the purpose of this test is not to measure an approximation error, but to compare visually the local interelement behaviour of a quantity derived from the approximate displacement. The convexity of the domain $\Omega$ implies that the displacement solution to this problem is expected to satisfy $u\in H^2(\Omega)$. Therefore, $\GRAD u\in H^1(\Omega)^2$, and a good approximation of the displacement should have a reconstructed gradient with small jumps across the mesh edges.

We apply the local indicator \eqref{eq.jump.discrete.gradient} to $\bvec q_h=P_hG_h\dofh{u}$, where $G_h\dofh{u}$ denotes the discrete displacement gradient and $P_hG_h\dofh{u}$ its natural potential reconstruction for the method under consideration. Figure~\ref{fig:grad_jump_visu} displays the results obtained for both methods, with $k=0$ and $t=1$ on a hexagonal mesh. The DDR--BGG reconstruction displays visibly smoother behaviour across the mesh skeleton, whereas the DDR--HHO reconstruction exhibits pronounced discontinuities around the region where the load changes sign, with a maximum local scaled jump of approximately $0.89$ for DDR--BGG compared with $17.65$ for DDR--HHO, corresponding to a reduction by a factor of about $20$. This enhanced regularity comes at the cost of additional degrees of freedom: on the mesh considered here, the DDR--HHO and DDR--BGG discrete spaces contain 3760 and 9882 degrees of freedom, respectively. This experiment illustrates qualitatively how the enhanced continuity inherited from the discrete $H^2$ structure improves the behaviour of quantities derived from the displacement.
\begin{figure}
  \centering

  \begin{minipage}{0.45\textwidth}
    \centering
    \includegraphics[height=7cm]{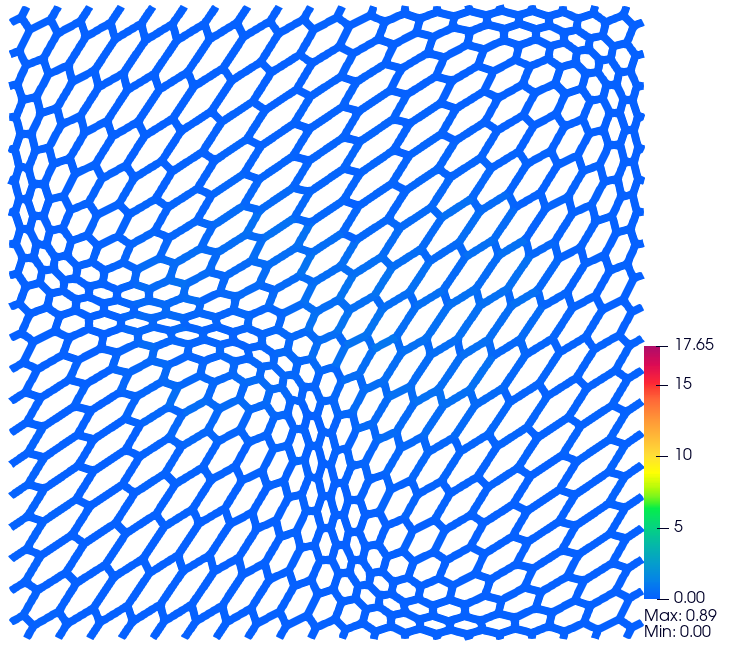}
    \subcaption{DDR--BGG}
  \end{minipage}
  \hfill
  \begin{minipage}{0.45\textwidth}
    \centering
    \includegraphics[height=7cm]{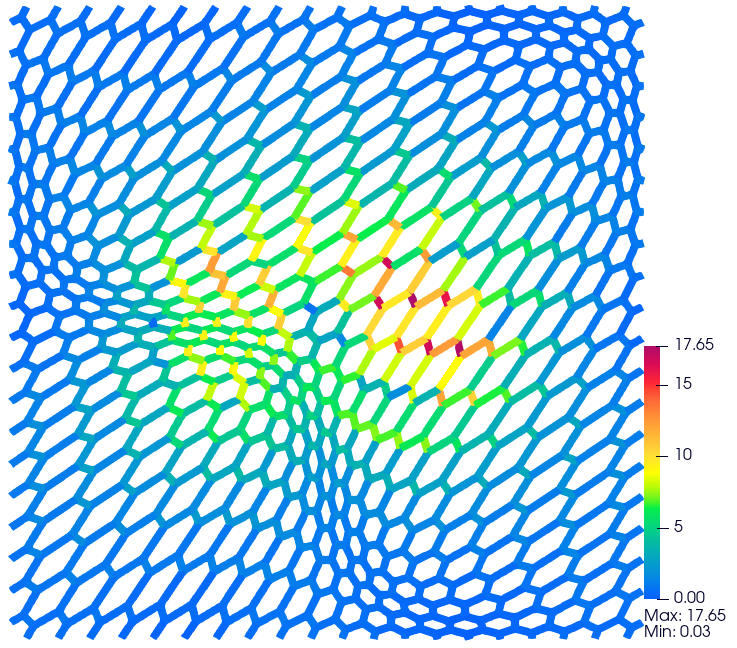}
    \subcaption{DDR--HHO}
  \end{minipage}

  \caption{Local scaled jumps $J_{1,E}(P_hG_h\dofh{u})$ see \eqref{eq.jump.discrete.gradient}) of the reconstructed discrete displacement gradient.}
  \label{fig:grad_jump_visu}
\end{figure}

\section{Thin-plate limit of the discrete schemes}
\label{sec:conv.models}

We now formally study the behaviour of the discrete Reissner--Mindlin scheme \eqref{eq:weak.rm}, for the two choices \eqref{eq:ah.hl} and \eqref{eq:ah.modified}, as the plate thickness $t$ tends to zero while the meshsize $h$ remains fixed. For each $t>0$, let
$(\dofh{\bvec\psi},\dofh{u})$ denote the corresponding discrete solution, with the dependence on $t$ left implicit. As the load $f$ is independent of $t$, a standard energy estimate for the scheme shows that its discrete energy is uniformly bounded with respect to $t$. Since the shear contribution to the energy is multiplied by $t^{-2}$, this boundedness yields
\begin{equation*}
  \norm{1,h}{\SGRAD{h}\dofh{u}-\dofh{\bvec\psi}} \lesssim t.
\end{equation*}
Consequently, as in the continuous setting,
\[
 \SGRAD{h} \dofh{u} - \dofh{\bvec\psi} \longrightarrow 0.
\]
For the choice \eqref{eq:ah.hl}, this formal limit leads to the following scheme: find $\dofh{u}\in\underline{H}_{2,0}^k(\Th)$ such that
\begin{equation}
\label{eq:discrete.kl}
  (\Hess{h} \dofh{u}, \Hess{h} \dofh{q} )_{\bvec{C},\VROT,h} = ( f, \SPgradh\dofh{q} )_{L^2(\Omega)}\qquad \forall \dofh{q}\in\underline{H}_{2,0}^k(\Th),
\end{equation}
where 
\begin{equation*}
  \underline{H}_{2,0}^k(\Th) \coloneq \left\{
  \dofh{q}\in \XSgrad{\Th} \,:\,  \ q_E = G_{q,E}^n = 0 \quad \forall  E \in \Eh^{\rm b},\,%
  \text{$q_V=0$ and $\bvec{G}_{q,V}=\bvec{0}$} \quad \forall V \in \Vh^{\rm b} \right\},
\end{equation*}
is a discrete counterpart of $H_{2,0}(\Omega)$, and the discrete Hessian operator is $\Hess{h}\coloneq \tGRAD{h}\SGRAD{h}$. This space appears because $\dofh{u}\in\XSgrad{\Th}$ and $\SGRAD{h}\dofh{u}\in\tXdRgrado{\Th}$ (as $\SGRAD{h}\dofh{u}= \dofh{\bvec \psi}$ imposes $ G_{u,E}^n = 0 $ for all $E\in\Eh^b$ and $\bvec{G}_{u,V}=\bvec{0}$ for all $V\in\Vh^b$).
When $\bvec C=\bvec\Id$ and no serendipity reduction of the cell unknowns is performed, the scheme \eqref{eq:discrete.kl} coincides with the Kirchhoff--Love scheme introduced in \cite{Di-Pietro.Droniou.ea:26.1}.
For the choice \eqref{eq:ah.modified}, the discrete scheme is a modified scheme for the Kirchhoff--Love model that can be proved to have the same convergence rate as the previous one using the techniques of \cite{Di-Pietro.Droniou.ea:26.1}. Find $\dofh{u}\in\underline{H}_{2,0}^k(\Th)$ such that
\begin{align*}
  (\bvec C \Hessh \dofh{u}, \Hessh \dofh{v} )_{L^2(\Omega)^{2\times 2 }} + \beta_0\sum_{T\in\Th} h_{T}^{-2} s_{1,T}(\SGRAD{T}\dofT{u},\SGRAD{T}\dofT{v}) = &{}( f, \SPgradh\dofh{v} )_{L^2(\Omega)},\\
  &{}\forall \dofh{v}\in\underline{H}_{2,0}^k(\Th).
\end{align*}

This relation between the discrete Reissner--Mindlin and Kirchhoff--Love schemes highlights the link between the discrete twisted complex \eqref{twisted-elasticity-boundary} and the discrete Hessian complex of \cite[Eq.~(6.1)]{Di-Pietro.Droniou.ea:26}, as already observed at the continuous level in \cite{Cap.Hu:24}.

\section*{Acknowledgements}

The author acknowledges the funding of the European Union via the ERC Synergy, NEMESIS, project number 101115663.

Views and opinions expressed are however those of the author only and do not necessarily reflect those of the European Union or the European Research Council Executive Agency. Neither the European Union nor the granting authority can be held responsible for them.

\appendix

\section{Notation}\label{appendix:notations}

  \begin{center}
    \renewcommand{\arraystretch}{1.2}
    \begin{longtable}{ccc}
      \toprule
      \textbf{Symbol} & \textbf{Continuous counterpart} & \textbf{Definition} \\
      \midrule
      \multicolumn{3}{c}{Discrete spaces} \\
      \midrule
      $\XSgrad{\Th}$
      & $H_2(\Omega)$
      & \eqref{eq:def.Xhess} \\
      $\XSrot{\Th}$
      & $H_1(\Omega)^2$
      & \eqref{eq:def.Xgrad} \\
      $\tXdRrot{\Th}$
      & $\boldsymbol{H}_{\VROT}(\Omega)$
      & \eqref{eq:def.Xrot} \\
      \midrule
      \multicolumn{3}{c}{Interpolators} \\
      \midrule
      $\ISgrad$
      & ---
      & \eqref{eq:def.ISgrad} \\
      $\IdRgrad$
      & ---
      & \eqref{eq:def.IdRgrad} \\
      $\IdRrot$
      & ---
      & \eqref{eq:def.IdRrot} \\
      \midrule
      \multicolumn{3}{c}{Discrete differential operators} \\
      \midrule
      $\SGRAD{h}$
      & $\GRAD : H_2(\Omega) \to H_1(\Omega)$
      & \eqref{eq:def.nablah}  \\
      $\tGRAD{h}$
      & $\GRAD : H_1(\Omega)^2 \to \boldsymbol{H}_{\VROT}(\Omega)$
      & \eqref{eq:def.uGh1} \\
      \midrule
      \multicolumn{3}{c}{Discrete potentials} \\
      \midrule
      $\SPgradh$
      & $q\in H_2(\Omega) $
      &  \eqref{eq:def.potential.2} \\
      $\Pgrad{k+2}{h}$
      & $\bvec v \in H_1(\Omega)^2 $
      & \eqref{eq:def.potential.1} \\
      $ \Prot{k+1}{h}$
      &  $\bvec{\xi} \in \boldsymbol{H}_{\VROT}(\Omega)$
      & \cite[Sec.~4]{Di-Pietro.Droniou:23} \\
      \bottomrule
    \end{longtable}
\end{center}


\printbibliography


\end{document}